\documentclass{amsart}
\usepackage{amsmath,amscd,amsthm,amsfonts,a4wide,amssymb}
\usepackage[all,cmtip]{xy}
\usepackage{fancyhdr}
\theoremstyle{plain}
\newtheorem{theorem}                 {Theorem}      [section]
\newtheorem{proposition}  [theorem]  {Proposition}

\newtheorem{lemma}        [theorem]  {Lemma}

\theoremstyle{definition}
\newtheorem{example}      [theorem]  {Example}
\newtheorem{remark}       [theorem]  {Remark}
\newtheorem{definition}   [theorem]  {Definition}

\newcommand{\be}[1]{\begin{equation}\label{#1}}
   \newcommand{\ee}{\end{equation}}

\numberwithin{equation}{section}

\DeclareMathOperator{\ad}{ad}

\DeclareMathOperator{\Ima}{Im} %Image

\DeclareMathOperator{\I}{Im} %Imaginary part

\DeclareMathOperator{\spa}{span}
\DeclareMathOperator{\End}{End}
\DeclareMathOperator{\Hom}{Hom}
\DeclareMathOperator{\rank}{rank}

\newcommand{\wt}{\widetilde}
\newcommand{\wh}{\widehat}
\newcommand{\cc}{\mathbb C}  %Complexification superscript
\newcommand{\pa}{\partial}

\def \nn{\mathbb N}
\def \zn{\mathbb Z}
\def \rn{\mathbb R}
\def \cn{\mathbb C}

\def \T{\mathcal T}
\def \W{\mathcal W}

\def \H{\mathbb H}
\def \O{\mathbb O}

\def \l{\ell}

\def \A{A^{\varphi}_z}

\def \HH{\mathcal H}

\def \g{\mathfrak{g}}
\def \h{\mathfrak{h}}
\def \k{\mathfrak{k}}

\def \p{\mathfrak{p}}
\def \q{\mathfrak{q}}
\def \t{\mathfrak{t}}

\def \SO#1{\mathrm{SO}(#1)}
\def \so#1{\mathfrak{so}(#1)}
\def \U#1{{\mathrm U}(#1)}
\def \u#1{\mathfrak{u}(#1)}

\def \su#1{\mathfrak{su}(#1)}
\def \Sp#1{\mathrm{Sp}(#1)}

\def \Orthog#1{{\mathrm O}(#1)}

\def \CP#1{\mathbb{C}P^{#1}}
\def \RP#1{\mathbb{R}P^{#1}}
\def \HP#1{\mathbb{H}P^{#1}}

\def \d{\mathrm{d}}
\def \pp{\mathfrak{p}}

\newcommand{\zbar}{\bar{z}}
\newcommand{\CC}{\underline{\mathbb C}}
\newcommand{\ov}{\overline}
\newcommand{\ul}{\underline}
\newcommand{\alg}{{\rm alg}}

\newcommand{\Hh}{\mathcal{H}} %horizontal bundle
\newcommand{\Vv}{\mathcal{V}} %vertical bundle

\newcommand{\eu}{\mathrm{e}} %Euler's number
\newcommand{\ii}{\mathrm{i}} %sqrt{-1}
\begin{document}
\baselineskip 18pt \larger[1]
\begin{footnotesize}
\begin{flushright}
CP3-ORIGINS-2013-11 DNRF90 \& DIAS-2013-11; J.\ London Math.\ Soc (to appear). 
\end{flushright}
\end{footnotesize}
\bigskip

\title{Harmonic maps into the exceptional symmetric space $G_2/SO(4)$}

\author{Martin Svensson} 
\author{John C. Wood}



\thanks{The first author was supported by the Danish Council for Independent Research under the project \emph{Symmetry Techniques in Differential Geometry}.
The second author thanks the Department of Mathematics and Computer Science of the University of Southern Denmark, Odense, for support during the preparation of this work.}

\address
{Department of Mathematics \& Computer Science, University of
Southern Denmark,
Campusvej 55, DK-5230 Odense M, Denmark}
\email{svensson@imada.sdu.dk}

\address
{Department of Pure Mathematics, University of Leeds, Leeds LS2 9JT, Great Britain}
\email{j.c.wood@leeds.ac.uk}

\begin{abstract}  
We show that a harmonic map from a Riemann surface into the exceptional
symmetric space $G_2/{\mathrm SO}(4)$ has a $J_2$-holomorphic twistor
lift into one of the three
flag manifolds of $G_2$ if and only if it is `nilconformal',
i.e., has nilpotent derivative.
  Then we find relationships with almost complex maps from a surface into the
6-sphere; this enables us to construct examples of nilconformal harmonic maps into $G_2/{\mathrm SO}(4)$ which are not of finite uniton number, and which have lifts into any of the three twistor spaces.  
Harmonic maps of finite uniton number are all nilconformal; for such maps, we show that our lifts can be constructed explicitly from extended solutions.
\end{abstract}

\maketitle

\section{Introduction}

\emph{Harmonic maps} are smooth maps between Riemannian manifolds which extremize the Dirichlet energy integral (see, for example, \cite{eells-lemaire}).
Harmonic maps from surfaces into symmetric spaces are of particular interest both to geometers, as they include \emph{minimal surfaces}, and to theoretical physicists, as they constitute the \emph{non-linear $\sigma$-model of particle physics}. 
Twistor methods for finding such harmonic maps have been around for some time: a general theory was given by F.~E.\ Burstall and J.~H.\ Rawnsley  \cite{burstall-rawnsley}.  The idea is to find a \emph{twistor fibration (for harmonic maps)}: this is a fibration $Z \to N$ from an almost complex manifold $Z$, called a \emph{twistor space}, to a Riemannian manifold $N$, with the property that
(almost-)holomorphic maps from Riemann surfaces to $Z$ project to harmonic maps into $N$.
For a symmetric space $N =G/H$, the theory of \cite{burstall-rawnsley} provides twistor spaces which are flag manifolds of $G$ precisely when $N$ is \emph{inner}, i.e., has inner Cartan involution; those twistor spaces come equipped with a canonical non-integrable complex structure $J_2$ and canonical fibration to $N$. 

We shall concentrate on the exceptional symmetric space $G_2/\SO{4}$; this has precisely three
twistor spaces $T_s$ \ $(s=1,2,3)$ which are flag manifolds of $G_2$,
see \S \ref{subsec:twistor-sp1}.  

A smooth map from a Riemann surface to a symmetric space is called \emph{nilconformal}
if its derivative is nilpotent, see Definition \ref{def:nilconformal} below;
here, by \emph{Riemann surface} we mean a connected (not necessarily compact) one-dimensional complex manifold.
Our main result is as follows:

\begin{theorem}\label{th:main}
A harmonic map from a Riemann surface to $G_2/\SO{4}$ has a twistor lift into one of the three twistor spaces $T_s$ if and only if it is nilconformal.
\end{theorem}

This was previously known only for harmonic maps from the $2$-sphere \cite{burstall-rawnsley}; these are all of finite uniton number \cite{uhlenbeck}. 
  The class of nilconformal harmonic maps includes those of finite uniton number, but is strictly bigger, see Example \ref{ex:torus};
it also includes \emph{inclusive maps}, see \S \ref{subsec:inclusive}, these have lifts into $T_1$. 

The method is as follows.
We embed $G_2/\SO{4}$ in $G_2$ and thence in $\U{7}$.
 Then, as in Uhlenbeck, the derivative of a harmonic map $\varphi:M \to G_2/\SO{4}
\subset \U{7}$
defines (on each coordinate domain) a holomorphic endomorphism $\A$ of a trivial
complex bundle $M \times \cn^7$ endowed with a holomorphic structure defined by $\varphi$.
Nilconformality means that this endomorphism is nilpotent; our twistor lifts are constructed from it, see Theorem \ref{th:lifts}.

In \cite{J2}, the authors showed that a harmonic map from a Riemann surface to a compact classical inner symmetric space has a twistor lift if and only if it is nilconformal --- previously known for maps into
a complex Grassmannian \cite{burstall-grass}.
The symmetric space $G_2/\SO{4}$ is the space of all
associative $3$-dimensional subspaces of $\rn^7$, so that it can be embedded in the real Grassmannian $G_3(\rn^7)$.
By realizing the twistor spaces $T_s$ as flags over this Grassmannian which
satisfy a $G_2$-condition, we interpret
our work in terms of the constructions in \cite{J2}, see
\S \ref{subsec:twistor-excep}.

In \S \ref{sec:almost-complex}, we show how to construct harmonic maps into
 $G_2/\SO{4}$ from almost complex maps into the $6$-sphere, giving harmonic maps with lifts into any of the three twistor spaces.

In \S \ref{sec:finite-uniton}, for harmonic maps of finite uniton number, we show that our twistor lifts are given by the \emph{canonical lift} of the authors \cite{J2}, see Theorem \ref{th:fi-uniton};
 in the case of maps with $S^1$-invariant extended solutions, the canonical lifts are \emph{superhorizontal},
and our formulae agree with those of \cite{correia-pacheco-G2}.  We end that section with two explicit examples of harmonic maps of finite uniton number which do not have $S^1$-invariant extended solutions.
 
In the last section, we will justify some of our assertions with the Lie theory in \cite{burstall-rawnsley}. 

The authors thank Fran Burstall, Luis Fern\'andez and Rui Pacheco for some
illuminating conversations, and the referee for three pertinent comments.

\section{Preliminaries} \label{sec:prelims}

\subsection{The Octonions and the action of $G_2$} \label{subsec:octonians}

Let $\H \cong \rn^4$ denote the quaternions and $\I\H \cong \rn^3$ 
the imaginary quaternions, both with their canonical orientations.
The algebra of \emph{octonions} $\O$, or \emph{Cayley numbers}, 
\cite[p.~113{\it ff}]{harvey-lawson} is an alternative real division algebra given by $\O=\H\oplus\H\cdot e$ for some unit octonion $e$.
Both $\O$ and its imaginary part $\I\O = \I\H \oplus \H \cdot e$ acquire orientations from those of $\H$ and $\I\H$.

Octonion multiplication $\cdot$ on $\O \cong \rn^8$ induces a vector product
 $\times$ on $\I\O \cong \rn^7$ by
$v \times w =$ the imaginary part of $v \cdot w$.  The \emph{real part}
 of $v \cdot w$ gives a positive definite inner product $( \cdot, \cdot)$
on $\I\O$; we extend both of these by complex bilinearity to
$\I\O \otimes \cn \cong \cn^7$.  The \emph{Hermitian inner product of}
$v,w \in \I\O \otimes \cn$ is given by $(v,\ov w)$: here, if $w = a + b\ii$ \ $(a,b  \in \I\O)$, then $\ov w$ denotes its complex conjugate $a - b\ii$.
If $(v,\ov w) = 0$, we call $v$ and $w$ \emph{orthogonal}, written $v \perp w$. 

The vector product is bilinear and enjoys the following properties for $u,v,w\in\I\O \otimes \cn$ \cite{harvey-lawson}:
\begin{eqnarray} 
(u,v \times w) &=& (u \times v, w), \label{vp-scalar}\\
u\times(v\times w)+(u\times v)\times w &=& 2(u,w)v- (u,v)w-
(v,w)u\,.  \label{vp-assoc}
\end{eqnarray}

A $3$-dimensional subspace $\xi\subset\I\O$ is said to be \emph{associative}
if it is the imaginary part of a subalgebra isomorphic to the quaternions.
It acquires a canonical orientation from that of the quaternions. In the sequel,
we write $\ov L_j$ to mean $\ov{L_j}$, etc. The following is left to the reader.

\begin{lemma} \label{le:assoc}
Let  $\xi$ be a $3$-dimensional subspace of\/ $\I\O$.

{\rm (a)} The following are equivalent$:$
\begin{enumerate}
\item[(i)] $\xi$ is associative\/$;$ 

\item[(ii)] $\xi$ is closed under the vector product\/$;$

\item[(iii)] $\xi$ has an orthonormal basis $\{e_1,e_2,e_3\}$ with
$e_1 \times e_2 = e_3\,;$

\item[(iv)] $\xi \otimes \cn$ has a basis $\{\ov L_1, L_0, L_1\}$ with
 $L_0 = \ov L_0$ and $L_1 \times \ov L_1 = \ii L_0\,;$

\item[(v)] $\xi \otimes \cn$ has a basis $\{\ov L_1, L_0, L_1\}$ with
$L_0 = \ov L_0$ and $L_0 \times L_1 = \ii L_1$\,.
\end{enumerate}

{\rm (b)} Let $\xi$ be associative.  Then {\rm (i)} its canonical orientation is that given by $\{e_1,e_2,e_1 \times e_2\}$ for any linearly independent elements $e_1,e_2$ of\/ $\xi;$
{\rm (ii)} for \emph{any} oriented orthonormal basis
$\{e_1,e_2,e_3\}$ we have $e_1 \times e_2 = e_3\,;$
{\rm (iii)} for any non-zero isotropic vector $L_1 \in \xi \otimes \cn$ of unit norm,
$\xi \otimes \cn$ has a basis $\{\ov L_1, L_0, L_1\}$ satisfying
{\rm (a)(iv)} and {\rm (a)(v)}.
\qed \end{lemma}

The group $G_2$ is the automorphism group of $\O$; it stabilizes $1$, and
since the scalar product is the real part of the octonian multiplication
$\cdot$\,, it acts by isometries on $\I\O$.  Since it also preserves the vector product, it preserves orientation so that $G_2\subset\SO{\I\O}\cong\SO{7}$. The induced action of $G_2$ on the Grassmannian of associative $3$-dimensional subspaces of $\I\O \cong \rn^7$ is transitive \cite[p.~114]{harvey-lawson}, and the stabilizer of $\xi=\I\H$ is $\SO{4}=\Sp{1}\times_{\zn_2}\Sp{1}$ so that $G_2/\SO{4}$ is the set of all associative $3$-dimensional subspaces of
$\I\O \cong \rn^7$.  Since such a subspace has a canonical orientation
(see Lemma \ref{le:assoc}), there is an embedding of $G_2/\SO{4}$ in the Grassmannian $\wt G_3(\rn^7)$ of \emph{oriented} $3$-dimensional subspaces of
$\I\O \cong\rn^7$.  

Let $J$ be the orthogonal complex structure
on $\xi^\perp$ with $(1,0)$-space $W$ and choose an orthonormal basis
$\{e_1,\dots,e_4\}$ of $\xi^\perp$ with $Je_1=e_2$ and $Je_3=e_4$.
We say that $J$, or the corresponding $W$, is \emph{positive} (resp.\
\emph{negative}) according as the basis is positive (resp.~negative). 

On the other hand, following \cite{correia-pacheco-G2}, we call a complex $2$-dimensional subspace $W$ \emph{complex-coassociative} if
$v \times w = 0$ for all $v,w \in W$; such a subspace is automatically isotropic, see below. These notions are linked as follows.

\begin{lemma}\label{le:coassoc-neg} Let $\xi$ be an associative $3$-dimensional
 subspace\/ $\I\O$ and let  $W$ be a maximally isotropic
subspace of\/ $\xi^{\perp} \!\otimes \cn$.  Then $W$ is positive
if and only if it is complex-coassociative.
\end{lemma}

\begin{proof} There is an orthornormal basis $\{e_i\}$ of $\xi^{\perp}$
with $W$ spanned by
 $v=e_1-\ii e_2$ and $w=e_3-\ii e_4$; a calculation gives 
$
(v\times w\,,\,\ov{v\times w}) = 4-2(e_1\,,\,[e_2,e_3,e_4]), 
$
where $[\cdot,\cdot,\cdot]$ is the associator defined by
$[u,v,w] = (u\cdot v)\cdot w - u \cdot (v \cdot w) $
\cite[p.~114]{harvey-lawson}. Now, by \cite[Theorem 1.16]{harvey-lawson},
$\{e_1,\dots,e_4\}$ is a positive basis for $\xi^\perp$ if and only if
$(e_1, [e_2,e_3,e_4])=2$;  the result follows. 
\end{proof}

\subsection{The standard representation of $G_2$ and the vector product on
$\rn^7$} \label{subsec:repn} 
We shall describe the representation of $G_2$ on $\rn^7$ in terms of its weight spaces, see \cite[\S 22.3]{fulton-harris}.
The Lie algebra $\g_2$ has simple roots $\alpha_1$ and $\alpha_2$, and the remaining roots are given by 
\begin{equation*}
\alpha_3=\alpha_1+\alpha_2\,, \quad
\alpha_4=2\alpha_1+\alpha_2\,, \quad
\alpha_5=3\alpha_1+\alpha_2\,, \quad
\alpha_6=3\alpha_1+2\alpha_2
\end{equation*}
and the negatives of the $\alpha_j$. The root lattice equals the weight lattice; the maximal root is
$3\alpha_1+2\alpha_2$. The representation $\I\O\otimes_\rn\cn\cong\cn^7$ of $\g_2^{\cc} = \g_2 \otimes \cn$ has weights
$\pm\alpha_1$, $\pm(\alpha_1+\alpha_2)$, $\pm(2\alpha_1+\alpha_2)$ and $0$.
For each weight $\lambda$, the corresponding weight space $\l_{\lambda}$ is
 one-dimensional with $\ov \l_{\lambda}=\l_{-\lambda}$, and is isotropic
for $\lambda\neq0$.  Clearly, we have
\be{weight-mult}
\l_{\lambda}\times \l_{\eta}\subset \l_{\lambda+\eta}\,,
\ee
where we set $\l_{\nu} =0$ if $\nu$ is not a weight. 
(For a precise multiplication table, see Example \ref{ex:torus}.) 

For any subspace
$\beta$ of $\I\O\otimes\cn$, the \emph{annihilator of $\beta$} is the subspace
$\beta^a = \{L \in \I\O\otimes\cn : L \times \beta = 0\}$.  If $\beta$ is isotropic of dimension one, then $\beta^a$ is isotropic of dimension three, and contains
$\beta$. To see this, it suffices to calculate $\beta^a$ when $\beta$ is a
weight space, see \cite{correia-pacheco-G2}; that this suffices follows from the transitivity of $G_2$ on isotropic unit vectors: in fact, by \cite[p.~115]{harvey-lawson},

\begin{lemma}\label{le:G2-trans}
The action of $G_2$ is transitive on \emph{pairs} $(u,v)$ of isotropic, orthogonal unit vectors with
$u\times v=0$. 
\qed
\end{lemma}

Using this, we see that any complex-coassociative subspace is isotropic, as it is $G_2$-equivalent to the
isotropic subspace $\l_{\alpha_1+\alpha_2} \oplus \l_{2\alpha_1+\alpha_2}$\,.

\begin{lemma} \label{le:expl-W}
Let $\xi$ be an associative $3$-dimensional subspace of\/ $\I\O$.

{\rm (a)} Let $\beta$ be a $1$-dimensional isotropic subspace of\/
$\xi^{\perp} \!\otimes \cn$.
Then
\begin{enumerate}
\item[(i)] the unique positive (equivalently, complex-coassociative) maximally isotropic subspace of\/ $\xi^{\perp} \!\otimes \cn$
which contains $\beta$ is given by $W = \beta^a \cap (\xi^{\perp} \!\otimes \cn);$

\item[(ii)] the unique negative (equivalently, not complex-coassociative) maximally isotropic subspace of\/ $\xi^{\perp} \!\otimes \cn$ which contains
$\beta$ is given by
$W = \ov\beta^{\perp} \cap \ov\beta^a \cap (\xi^{\perp} \!\otimes \cn)$.
\end{enumerate}

{\rm (b)}  Let $\beta$ be a $2$-dimensional positive isotropic subspace of\/ $\xi^{\perp} \!\otimes \cn$.  Then
$\beta = \beta^a \cap (\xi^{\perp}\!\otimes\cn)$.
\end{lemma}

\begin{proof}
By Lemmas \ref{le:G2-trans} and \ref{le:assoc}, we can take
$\xi \otimes \cn = \l_{-\alpha_1} \oplus \l_0 \oplus \l_{\alpha_1}$ and
$\beta = \l_{2\alpha_1+\alpha_2}$.
Then the two maximally isotropic subspaces of $\xi^{\perp}$ containing $\beta$ are
$\l_{\alpha_1+\alpha_2} \oplus \l_{2\alpha_1+\alpha_2}$, which is
 complex-coassociative, and
$\l_{-\alpha_1-\alpha_2} \oplus \l_{2\alpha_1+\alpha_2}$, which is not.
The formulae follow.
\end{proof}

\subsection{The twistor spaces of $G_2/\SO{4}$} 
\label{subsec:twistor-sp1}

Let $N$ be a Riemannian manifold.
By a \emph{twistor fibration of\/ $N$ (for harmonic maps)} is meant
\cite{burstall-rawnsley} an almost complex manifold $(Z,J)$
(called a \emph{twistor space}) 
 and a fibration $\pi:Z \to N$ such that, for every (almost-)holomorphic map from a Riemann surface
$\psi:M \to Z$, the composition $\varphi = \pi \circ \psi:M \to N$ is harmonic.
Then $\varphi$ is called the  \emph{twistor projection of $\psi$}, and $\psi$ is called \emph{a twistor lift of $\varphi$}. 

For inner symmetric spaces, a general theory of such
twistor fibrations is given in \cite{burstall-rawnsley}.  The twistor spaces are
flag manifolds $G/H$ where $H$ is the centralizer of a torus.  

Now $\so{4}=\su{2}\oplus\su{2}$, and a maximal torus is given by
$\u{1}\oplus\u{1}$. 
There are two more possibilities for centralizers of tori in $\so{4}$, see \S \ref{subsec:twistor-fibr2}, namely $\u{2}_+=\su{2}\oplus\u{1}$ and $\u{2}_-=\u{1}\oplus\su{2}$.
Hence, there are precisely three flag manifold of $G_2$ fibring canonically over
$G_2/\SO{4}$, which we shall denote by $T_s$ \ $(s=1,2,3)$ as in the following commutative diagram from \cite[\S 11.9]{salamon-book}; the fibration $\pi_6$
will be explained in \S \ref{sec:almost-complex}.
\be{diag:twistor-spaces}
\xymatrix{
& T_3 = G_2/\U{1}\times\U{1}\ar[dl]_{\pi_4}\ar[dr]^{\pi_5}\ar[dd]^{\pi_3} &\\
T_1 = G_2/\U{2}_+\ar[rd]_{\pi_1} & & T_2 = G_2/\U{2}_-\ar[dl]^{\pi_2}
\ar[dr]_{\pi_6} \cong Q^5\\
& G_2/\SO{4} & & S^6
}
\ee

Note that every fibre, apart from those of $\pi_3$, is isomorphic to $\cn P^1$.
The $\pi_s$  \ $(s=1,2,3)$ are fibre bundles over $G_2/\SO{4}$ associated to the principal $G_2$-bundle; the following geometric description will be useful to us --- see \S \ref{subsec:twistor-fibr2} for a Lie algebraic treatment.

$T_1 = G_2/\U{2}_+$ consists of all rank $2$ isotropic subspaces $W$ of
$\I\O \otimes \cn \cong \cn^7$ which are complex-coassociative.  The projection
$\pi_1:G_2/\U{2}_+ \to G_2/\SO{4}$ is given by $W \mapsto \{W \oplus \ov{W}\}^{\perp}$.
By Lemma \ref{le:coassoc-neg},
the fibre at any $\xi \in G_2/\SO{4}$ is  all positive maximally isotropic subspaces of $\xi^{\perp} \!\otimes \cn$, equivalently, all positive orthogonal complex structures on $\xi^{\perp}$; thus $T_1$ is the quaternionic twistor space of $G_2/\SO{4}$ \cite{salamon-quat}.

$T_2 = Q^5$, the complex quadric
$\bigl\{[z_0,\ldots,z_5] \in \CP{6} : z_0^{\,\,2} + \cdots + z_5^{\,\,2} = 0 \bigr\}$ consisting of all one-dimensional isotropic subspaces of $\I\O\otimes \cn$.
The projection $\pi_2:Q^5 \to G_2/\SO{4}$ is given by
$\pi_2(\l)= \l \oplus \ov\l \oplus (\l \times \ov\l) = $ the unique associative $3$-dimensional subspace containing $\l$ (cf.\ Lemma \ref{le:assoc}(b)(iii)).

Alternatively, we can think of $T_2$ as $G_2/\U{2}_-$\,, the space of all rank
$2$ isotropic subspaces $W$ of $\I\O\otimes \cn$ which are \emph{not} complex-coassociative.
There is a $G_2$-equivariant bundle isomorphism from
$Q^5$ to $G_2/\U{2}_-$ given by
$\l \mapsto \l^a \ominus \l$ with inverse $W \mapsto W \times W$
(here $\l^a \ominus \l$ denotes $\l^{\perp} \cap \l^a$).
The projection
$\pi_2:G_2/\U{2}_- \to G_2/\SO{4}$ is given by
$W \mapsto \{W \oplus \ov W\}^{\perp}$.
By Lemma \ref{le:coassoc-neg}, the fibre at any $\xi \in G_2/\SO{4}$ consists of 
all negative maximally isotropic subspaces of $\xi^{\perp} \!\otimes \cn$,
equivalently, all negative orthogonal complex structures on $\xi^{\perp}$.

$T_3 = G_2/(\U{1} \times \U{1})$ consists of all pairs $(\l,D)$ where $\l$ and $D$ are subspaces of $\I\O\otimes\cn$ of ranks $1$ and $2$,
respectively, which satisfy $\l \subset D \subset \l^a$.  The projection
$\pi_3:G_2/(\U{1} \times \U{1}) \to G_2/\SO{4}$ is given by
$\pi_3(\l,D) = q \oplus \ov q \oplus (q \times \ov q)$ where $q = D \ominus \l$.

$\pi_4$ (resp.\ $\pi_5$)$: G_2/(\U{1} \times \U{1}) \to G_2/\U{2}_{\pm}$
is given by $(\l, D) \mapsto$ the positive (resp.\ negative) maximally isotropic subspace of $\xi^{\perp} \times \cn$ which contains $\l$;
equivalently $\pi_5: G_2/(\U{1} \times \U{1}) \to Q^5$ is given by
$(\l, D) \mapsto D \ominus \l$.

For each twistor space, the theory of \cite{burstall-rawnsley} gives a decomposition of the tangent bundle into horizontal and vertical parts, and canonical almost complex structures $J_1$ and $J_2$ which agree on the horizontal space, see
\S \ref{sec:Lie-theory}.  It can be checked that the maps $\pi_4$ and $\pi_5$ in diagram \eqref{diag:twistor-spaces} map the horizontal spaces of $\pi_3$ to those of $\pi_1$ and $\pi_2$,
but do not intertwine $J_1$ or $J_2$.

\section{Harmonic maps and their twistor lifts}
\subsection{Harmonic maps into a Lie group} \label{subsec:Lie-groups}
Throughout the paper, all manifolds, bundles, and structures on them are
 taken to be $C^{\infty}$-smooth.  Recall that harmonic maps from surfaces enjoy conformal invariance (see \cite{wood60}), so that the concept of harmonic map from a Riemann surface $M$ is well defined. 
Let $G$ be a compact Lie group.
 For any smooth map $\varphi:M\to G$, set $A^{\varphi}=\frac{1}{2}\varphi^{-1}\d\varphi$;
thus $A^{\varphi}$ is a $1$-form with values in the Lie algebra $\g$ of $G$;
it is half the pull-back of the Maurer--Cartan form of $G$.
To study maps into a symmetric space $G/H$, we embed $G/H$ in the Lie group $G$ by the totally geodesic Cartan embedding (or immersion, see \cite[Proposition 3.42]{cheeger-ebin}); this preserves harmonicity.

Now, any compact Lie group can be embedded totally geodesically in the unitary group $\U n$, so we now consider that group together with its standard action on $\cn^n$.  Let $\CC^n$ denote the trivial complex bundle $\CC^n = M \times \cn^n$, then
$D^{\varphi} = \d+A^{\varphi}$
defines a unitary connection on $\CC^n$.  For convenience, we choose
a local complex coordinate $z$ on an open set $U$ of $M$, however, our key constructions will be independent of that choice and so globally defined.
We write
$\d\varphi = \varphi_z \d z + \varphi_{\zbar}\d\zbar$,
$A = A^{\varphi}_z \d z +  A^{\varphi}_{\zbar} \d\zbar$,
$D^{\varphi} = D^{\varphi}_z \d z + D^{\varphi}_{\zbar} \d\zbar$,
$\pa_z = \pa/\pa z$ and $\pa_{\zbar} = \pa/\pa\zbar$.  Then
\be{type-decomp}
A^{\varphi}_z=\frac{1}{2}\varphi^{-1}\varphi_z\,,\quad A^{\varphi}_{\zbar}=\frac{1}{2}\varphi^{-1}\varphi_{\zbar}\,, \quad
D^{\varphi}_z = \pa_z + A^{\varphi}_z \,,\quad
D^{\varphi}_{\zbar} = \pa_{\zbar} + A^{\varphi}_{\zbar} \,.
\ee

Interpreting $\A$ and $A^{\varphi}_{\zbar}$ as (locally defined)
endomorphisms of $\CC^n$,
the adjoint of $\A$ (with respect to the Hermitian inner product) is $-A^{\varphi}_{\zbar}$;
in particular, if $\varphi$ is real, i.e., maps into $\Orthog{n}$, then $\A$ is skew-symmetric, i.e., for any $p \in M$,
\be{antisym}
(\A(u), v) = -(u,\A(v)) \qquad (u,v \in \{p\} \times \cn^n)\,,
\ee
where $( \cdot, \cdot )$ denotes the standard symmetric bilinear inner product
on $\cn^n$.
We have some simple properties which follow from \eqref{type-decomp} and \eqref{antisym}:

\begin{lemma} \label{le:diff1} 
{\rm (i)} If $\beta$ is a holomorphic subbundle of\/
$(\CC^n,D^{\varphi}_{\zbar})$, so is its \emph{polar} $\ov\beta^{\perp}$.

{\rm (ii)}  If $\beta$ is an isotropic line subbundle of\/ $\CC^n$, then
$(D^{\varphi}_{\zbar}(\beta), \beta) = (A^{\varphi}_{\zbar}(\beta), \beta) = 
(D^{\varphi}_z(\beta), \beta) = (A^{\varphi}_z(\beta), \beta) = 0$.  
\end{lemma}

There is a unique holomorphic structure on $\CC^n$ with $\bar{\pa}$-operator given over each coordinate domain $(U,z)$ by $D^{\varphi}_{\bar z}$
\cite{koszul-malgrange}, we call this the \emph{(Koszul--Malgrange) holomorphic structure  induced by $\varphi$}; the resulting holomorphic vector bundle will be denoted by $(\CC^n, D^{\varphi}_{\bar z})$.
Then \cite{uhlenbeck} \emph{a smooth map $\varphi:M\to \U{n}$ is harmonic
if and only if, on each coordinate domain, $\A$ is a holomorphic endomorphism of the holomorphic vector bundle $(\CC^n, D^{\varphi}_{\bar z})$}.
For any holomorphic endomorphism $E$, at points where it
 does not have maximal rank, we shall `fill out zeros'  as in
\cite[Proposition 2.2]{burstall-wood} (cf.\ \cite[\S 3.1]{unitons})
to make its image and kernel into holomorphic subbundles $\Ima E$ (which we often denote simply by $E(\CC^n)$) and $\ker E$ of $\CC^n$.

For any $k,n$ with $0 \leq k \leq n$, let $G_k(\cn^n)$ (resp.\ $G_k(\rn^n)$) denote the Grassmannian of $k$-dimensional subspaces of $\cn^n$ (resp.\ $\rn^n$).   
The class of maps which we consider is given by the following generalization of a notion that F.~E.\ Burstall \cite{burstall-grass} defined for harmonic maps $M \to G_k(\cn^n)$.   A more general formulation is given in
\S \ref{subsec:nilconformal}.

\begin{definition} \label{def:nilconformal}
A smooth map $\varphi:M \to \U{n}$ from a Riemann surface is called \emph{nilconformal} if 
$(\A)^{r} = 0$ for some $r \in \nn$.  
We call the least such $r$ the \emph{nilorder of\/ $\varphi$}. %$($in $\U{n})$}.
\end{definition}

Note that our concept of nilconformal reduces to that in \cite{burstall-grass} for maps into a complex Grassmannian embedded in $\U{n}$ via the Cartan embedding, but our nilorder may differ by $\pm 1$ from that
in \cite{burstall-grass}.
\emph{Harmonic maps of finite uniton number are nilconformal} (see \cite[Example 4.2]{J2}), but the converse is false, see Example \ref{ex:torus} below.

For a non-constant nilconformal harmonic map
$\varphi:M \to G_k(\rn^n)$ or $G_k(\cn^n)$, \emph{let $s(\varphi)$
denote the least positive integer $s$ such 
$(\A)^{2s}\bigl((-1)^{s-1}\varphi\bigr) = 0$}; here $-\varphi$
denotes $\varphi^{\perp}$.
Note that $r-1 \leq 2s(\varphi) \leq r+1$, where $r$ is the nilorder of $\varphi$ as defined above.

To study harmonic maps into $G_2/\SO{4}$, we embed $G_2/\SO{4}$ in $G_2$ by the Cartan embedding, and then $G_2$ into $\SO{7}$ and finally into $\U{7}$.
Equivalently, embed $G_2/\SO{4}$ in the real Grassmannian $G_3(\rn^7)$, then into
the complex Grassmannian $G_3(\cn^7)$, and finally via the Cartan embedding
into $\U{7}$.  In particular, a harmonic map $M \to G_2/\SO{4}$ is nilconformal 
if it is nilconformal as a map into $\U{7}$.  We shall repeatedly use the following
simple rules.

\begin{lemma} \label{le:diff2}
{\rm (i)} Let $v, w:M \to \cn^7$ be smooth maps.  Then
$\A(v \times w) = \A(v) \times w \,+ \, v \times \A(w)$.
Similar rules hold for $A^{\varphi}_{\zbar}$, $D^{\varphi}_z$
and $D^{\varphi}_{\zbar}$.
 
{\rm (ii)} If\/ $\beta$ is a holomorphic subbundle of\/
$(\CC^7,D^{\varphi}_{\zbar})$, so is its annihilator $\beta^a$.

{\rm (iii)}  If\/ $\varphi:M \to G_2/\SO{4}$ is a harmonic map, and $\beta$ is a
holomorphic line subbundle of\/ $\varphi^{\perp}$,
then the subbundles\/ $W$ of\/ $\varphi^{\perp}$ given by the formulae
of Lemma \ref{le:expl-W}(a) are holomorphic.
\end{lemma}

\begin{proof}
(i) and (ii) Since $\varphi$ has values in $G_2$\,, $\A$ has values in
$\g_2 \otimes \cn$.   The rest follows.

(iii) This follows from Lemma \ref{le:diff1}(ii) and
$(D^{\varphi}_{\zbar}W, \beta) = (W, D^{\varphi}_{\zbar}\beta) = 0$.
\end{proof}

\subsection{Construction of the twistor lifts} \label{subsec:constr-lifts}
We may now state our main result, namely that nilconformal harmonic maps into $G_2/\SO{4}$ have twistor lifts; in fact, we shall give explicit formulae for those lifts.  For uniqueness see Lemma \ref{le:uniqueness}. The converse, with estimate on $s(\varphi)$, follows from Proposition \ref{pr:real-Grass}; for the general fact that a harmonic map with a twistor lift is nilconformal, see Proposition \ref{pr:nilconformal}.

Our proofs will need the following characterization of when a map from a Riemann surface $M$
 to one of the twistor spaces $T_s$ of \S \ref{subsec:twistor-sp1} is $J_2$-holomorphic; this will be established using Lie theory in \S \ref{subsec:J2-descr}. 
Those maps are given by varying subspaces $W$, $\l$ and $D$ of $\cn^7$,
i.e., subbundles of the trivial bundle
$\CC^7 = M \times (\I\O \otimes \cn) \cong M \times \cn^7$.   As before, we write
$(\CC^7,D^{\varphi}_{\zbar})$ to mean the bundle $\CC^7$ equipped with the Koszul--Malgrange holomorphic structure induced by $\varphi$.  

\begin{lemma} \label{le:J2-descr}
Let $\psi:M \to T_s$ be a smooth map $(s=1,2,3)$ and set $\varphi = \pi_s \circ \psi$.  Then $\psi$ is
$J_2$-holomorphic (so that $\varphi$ is harmonic) if and only if
\begin{eqnarray*}
(s=1) && \psi = W \text{ is a holomorphic subbundle of\/ $(\CC^7,D^{\varphi}_{\zbar})$ lying in $\ker\A;$} \nonumber \\
(s=2) && \psi = \l \text{ is a holomorphic subbundle of\/ $(\CC^7,D^{\varphi}_{\zbar})$ 
lying in $\ker\A;$} \nonumber \\
(s=3) &&  \psi = (\l, D) \text{ where $\l$ and $D$ are holomorphic
subbundles of\/ $(\CC^7,D^{\varphi}_{\zbar})$}\\
&&\text{with $\l \subset \ker\A$ and $\A(D) \subset \l$.}
\end{eqnarray*}

The same result holds for $J_1$-holomorphic if we replace `holomorphic subbundle' by `antiholomorphic subbundle' throughout.
\end{lemma} 

\begin{theorem}\label{th:lifts} Let $\varphi:M \to G_2/\SO{4}$ be a non-constant nilconformal harmonic map.  Then $s(\varphi) \in \{1,2,3\}$, and $\varphi$ has a unique $J_2$-holomorphic lift from $M$ to $T_s$ where $s = s(\varphi)$. 
\end{theorem}

\begin{proof}
Suppose that $\varphi:M \to G_2/\SO{4}$ is a non-constant nilconformal
harmonic map.   We first note that, by nilpotency,
$\rank (\A)^{2i}(\varphi) \leq 3-i$ \ $(i=1,2,3)$,
in particular, $s(\varphi) \leq 3$.

We now construct a $J_2$-holomorphic lift into $T_s$ where $s = s(\varphi)$.
First, since $\A$ is holomorphic, we may fill out zeros as described in
\S \ref{subsec:Lie-groups} to obtain a holomorphic subbundle $\beta =  (\A)^s({(-1)^{s-1}\varphi})$ (i.e., $\beta =  \Ima\{(\A)^s|_{(-1)^{s-1}\varphi}\} )$ of $\varphi^{\perp}$. 
 Then $(\beta,\beta) = \bigl((\A)^{2s}({(-1)^{s-1}\varphi}),
(-1)^{s-1}\varphi \bigr)$, which is zero by definition of $s(\varphi)$;
hence \emph{$\beta$ is an isotropic holomorphic subbundle of\/ $\varphi^{\perp}$},
which is thus of rank $0$, $1$ or $2$.

Next let $W$ be a maximally isotropic holomorphic subbundle of $\varphi^{\perp}$ containing $\beta$,  then
\be{W-beta}
(\A)^s((-1)^{s-1}\varphi) = \beta \subset W \subset
\ov\beta^{\perp} = \ker (\A)^s|_{\varphi^{\perp}}\,,
\ee
the last equality following from \eqref{antisym}.
{}From \eqref{W-beta} and the definition of $s(\varphi)$ we see that
\be{last-leg}
(\A)^i(W) = 0 \quad \text{if and only if} \quad  i \geq s(\varphi).
\ee
When $s=3$, we will choose $W$ to satisfy the additional condition
\be{W-cond}
(\A)^2(W) \subset W
\ee
(which is automatic from \eqref{last-leg} for $s =1,2$).
Note that
\be{all-hol}
(\A)^i(W) \quad \text{is a holomorphic isotropic subbundle of}
\quad (-1)^{i+1}\varphi \quad (i=0,1,\ldots),
\ee
the isotropy following from
$\bigl(\A(W),\A(W)\bigr) = \bigl((\A)^2(W), W\bigr)$ which is zero
by \eqref{W-cond}; further, \emph{the last non-zero one, $(\A)^{s-1}(W)$, is in $\ker\A$}.
We shall build the lift of $\varphi$ from the subbundles $(\A)^i(W)$ \
$(0 \leq i\leq s-1)$; to obtain a lift into $T_s$, we will choose $W$ to be positive when
$s$ is odd and negative when $s$ is even. The proof is completed by the next three lemmas.
The space of smooth sections of a vector bundle will be denoted by $\Gamma(\cdot)$.
\end{proof}

\begin{lemma} \label{le:s1}
Let $\varphi:M \to G_2/\SO{4}$ be a non-constant nilconformal harmonic map with
$s(\varphi) = 1$. Then there is a $J_2$-holomorphic lift\/
$W:M \to T_1 = G_2/\U{2}_+$ of\/ $\varphi$ given by
\be{W-s1}
W  = \beta^a \cap \varphi^{\perp} \quad \text{where} \quad \beta = \A(\varphi).
\ee
\end{lemma}

\begin{proof}
First note that $\beta = 0$ would imply that $\varphi$ is constant, therefore, since
$\beta$ is isotropic in $\varphi^{\perp}$, it has rank $1$ or $2$.

{\rm (a)} If $\rank\beta = 1$, then by Lemma \ref{le:expl-W}, $\beta$ has a unique extension to a positive, equivalently complex-coassociative, maximally isotropic subbundle $W$ of $\varphi^{\perp}$ given by \eqref{W-s1},
and this is holomorphic by Lemma \ref{le:diff2}.
 
{\rm (b)} If $\rank\beta = 2$, set $W = \beta$.
We claim that $W$ is complex-coassociative.
To see this, let $v, w \in \Gamma(\varphi)$.
Since $\varphi$ is associative,
$v \times w \in \Gamma(\varphi)$.
{}From $s(\varphi) = 1$ we have $(\A)^2(\varphi) = 0$, hence
applying $\A$, by Lemma \ref{le:diff2}(i) we obtain
$\A(v) \times w + v \times \A(w) \in \Gamma(\A(\varphi))$.
Applying $\A$ again,  we obtain
$\A(v) \times \A(w) =0$.   Hence $W$ is complex-coassociative.
Further, by Lemma \ref{le:expl-W}(b), $\beta^a \cap \varphi^{\perp} = \beta = W$.

In both cases, we obtain a map $W:M \to T_1$; clearly,
$\varphi = (W \oplus \ov W)^{\perp}$, so $W$ is a lift of $\varphi$.
By Lemma \ref{le:J2-descr}, it is $J_2$-holomorphic and the lemma follows.
\end{proof}

\begin{lemma} \label{le:s2}
Let $\varphi:M \to G_2/\SO{4}$ be a non-constant nilconformal harmonic map with
$s(\varphi) = 2$. Then there is a $J_2$-holomorphic lift
$\l:M \to T_2 = Q^5$ of $\varphi$ given by
\be{W-s2}
\l = \A(W) \quad \text{where} \quad W = \ov\beta^{\perp} \cap \ov\beta^a \cap \varphi^{\perp} \text{ with } \beta = (\A)^2(\varphi^{\perp}).
\ee
\end{lemma}

\begin{proof}
{\rm (a)} Suppose first that $\rank\beta = 0$; then \eqref{W-s2} gives
$\l = \A(\varphi^{\perp})$. Then $\l$ is a non-zero holomorphic isotropic subbundle of $\varphi$, and is thus of rank $1$; further, it is in the kernel of $\A$.  It thus gives a $J_2$-holomorphic lift of\/ $\varphi$ into $Q^5$.
Note that taking $W$ to be
\emph{any} maximally isotropic subbundle of $\varphi^{\perp}$ and setting
$\l = \A(W)$ gives the same $\l$.

\smallskip

{\rm (b)}   Next, suppose that $\rank\beta \neq 0$.  If $\rank\beta=1$,
by Lemma \ref{le:expl-W}, it has a unique extension to a maximally isotropic subbundle $W$ of $\varphi^{\perp}$ which is negative, equivalently, not complex-coassociative, and this is holomorphic by Lemma \ref{le:diff2}.  If $\rank\beta = 2$, set $W = \beta$
(which may or may not be complex-coassociative).

In both cases, $\l = \A(W)$ is an isotropic line subbundle of $\varphi$. As in Lemma \ref{le:assoc}, $\varphi = \l \oplus \ov\l \oplus (\l\times\ov\l)$ so that $\l$  gives a lift of\/ $\varphi$ into $Q^5$\,;
by Lemma \ref{le:J2-descr} it is $J_2$-holomorphic.

Further, with notation which will be convenient in
\S \ref{subsec:twistor-excep}, set $\psi_2 = \l$,
$\psi_{-2} = \ov\psi_2$ and $\psi_0 = \psi_{-2} \times \psi_2$; then
$\varphi = \psi_{-2} \oplus \psi_0 \oplus \psi_2$ and, by Lemma \ref{le:assoc}(v),
$\psi_0 \times \psi_2 = \psi_2$.
Now  $\bigl((\A)^2(\varphi), \psi_2 \bigr) = \bigl((\A)^2(\varphi), \A(W)\bigr)
	= \bigl(\varphi, (\A)^3(W)\bigr) = 0$ so that
$(\A)^2(\varphi)$ lies in the polar of $\psi_2$, i.e.,  
\be{polar}
(\A)^2(\varphi) \subset \psi_0 \oplus \psi_2\,.
\ee

Similarly, $(\A(\varphi^{\perp}), \psi_2) = (\varphi^{\perp}, \A(\psi_2)) = 0$,
so that $\A(\varphi^{\perp}) \subset \psi_0 \oplus \psi_2$.  Hence
$\beta = (\A)^2(\varphi^{\perp}) \subset \A(\psi_0)$ and, since $\psi_0$
is of rank one, so is $\beta$, so that
an explicit formula for $W$ is given by Lemma \ref{le:expl-W}(a)(ii).  
The lemma follows.

Further, applying $\A$ to $\psi_0 \times \psi_2 \subset \psi_2$ gives
$\A(\psi_0) \times \psi_2 = 0$, so $\beta \times \l = 0$, i.e.,
$\beta \subset \l^a \cap \varphi^{\perp}$.  It follows that,
when $\rank\beta =1$, we have 
$W = \l^a \cap \varphi^{\perp}$ as this contains $\beta$ and is negative.
When $\rank\beta = 0$, as previously noted, $\l = \A(W)$ for any 
maximally isotropic subbundle $W$ of $\varphi^{\perp}$, but
it is convenient to take $W = \l^a \cap \varphi^{\perp}$ so that, in all cases,
$W \times W = \l$.  Thus $W$ gives the lift into $T_2$ thought of as
$G_2/\U{2}_-$\,, see also \S \ref{subsec:twistor-excep}.
\end{proof}

\begin{lemma} \label{le:s3}
Let $\varphi:M \to G_2/\SO{4}$ be a non-constant nilconformal harmonic map with
$s(\varphi) = 3$. Then there is a $J_2$-holomorphic lift 
$(\l,D):M \to T_3 = G_2 \big/(\U{1} \times \U{1})$ of\/ $\varphi$
given by
\be{W-s3}
\l = (\A)^2(W), \ D = \A(W) \oplus (\A)^2(W) %\quad 
\text{ where } %\quad
W = \beta^a \cap \varphi^{\perp}
\text{ with } \beta = (\A)^3(\varphi).
\ee
\end{lemma}

\begin{proof}  
(i) We first show that there is a maximally isotropic holomorphic subbundle $\wt W$ of $\varphi^{\perp}$ which contains $\beta$ and satisfies \eqref{W-cond}. 
 
Suppose that $\rank \beta = 0$.  Then $(\A)^4(\varphi^{\perp}) = 0$ so that $s(\varphi) \leq 2$, a contradiction.

Suppose, next, that $\rank \beta = 2$, then set $\wt W = \beta$;
this clearly satisfies \eqref{W-cond}.

Finally, suppose that
$\rank\beta =1$.  Since $\beta$ is closed under $(\A)^2$, so is its polar $\ov\beta^{\perp}$; hence $(\A)^2$ factors to an endomorphism $B$ of the rank $2$ quotient bundle $\ov\beta^{\perp}\!\!\big/\beta$.  Then, \emph{either} $B=0$, in which case we may define $\wt W$ to be $\beta + \gamma$ where $\gamma$ is either of the two isotropic line subbundles in $\ov\beta^{\perp}\!\!\big/\beta$ --- explicit
formulae are given by \ref{le:expl-W}(a)(i) or (ii)
 --- \emph{or} $B \neq 0$ and $B^2=0$, in which case we set
$\wt W = \beta + (\A)^2(\ov\beta^{\perp})$. 
In either case, $\wt W$ satisfies \eqref{W-cond}; also
$\A(\wt W)$ and $(\A)^2(\wt W)$ are isotropic, and non-zero by \eqref{last-leg}, and so both have rank one. 

\smallskip 

(ii) Using part (i) we now show that we can actually choose a $W$ which is positive, equivalently, complex-coassociative. There are two cases, as follows.

\smallskip

(a) Suppose that  $\A(\beta) = 0$.  Then $\rank\beta =1$, otherwise
 $s(\varphi) =1$ by \eqref{last-leg}.
As in Lemma \ref{le:expl-W}(a)(i), set $W = \beta^a \cap \varphi^{\perp}$.  Then $W$ is a positive holomorphic maximally isotropic subbundle of $\varphi^{\perp}$ which contains $\beta$.  Further, $W \times \beta = 0$.  Applying $\A$ twice gives $(\A)^2(W) \times \beta = 0$, showing that $W$ satisfies \eqref{W-cond}.

\smallskip

(b) Suppose, instead, that $\A(\beta) \neq 0$.  Let $\wt W$ be a maximally isotropic subbundle of $\varphi^{\perp}$ of either sign containing $\beta$ and satisfying \eqref{W-cond}, as constructed in part (i).    Since
$\A(\beta)$ is contained in $\A(\wt W)$ and they both have rank one, they must be
equal, whence $(\A)^2(\beta) = (\A)^2(\wt W)$, which is non-zero by \eqref{last-leg}.
This implies that $\rank\beta=2$; so set $W = \beta$, this satisfies
\eqref{W-cond}.  
To see that this $W$ is positive, as for the case $s=2$, we decompose
$\varphi = \psi_{-2} \oplus \psi_0 \oplus \psi_2$ where $\psi_2 = \A(W)$,
$\psi_{-2} = \ov\psi_2$ and $\psi_0 = \psi_{-2} \times \psi_2$; then
\eqref{polar} holds.  Let $\Psi \in \Gamma(\varphi)$ and $\Psi_2 \in \Gamma(\psi_2)$.
Then from \eqref{polar} and associativity of $\varphi$ (see Lemma \ref{le:assoc}(v)),
$(\A)^2(\Psi) \times \Psi_2 \in \Gamma(\psi_2)$.
Applying $\A$, we see that
$(\A)^3(\Psi) \times \Psi_2 + (\A)^2(\Psi) \times \A(\Psi_2) \in \Gamma(\A(\psi_2))$.
Now $(\A)^2(\Psi_2) = 0$ and $(\A)^4(\Psi) = 0$, so applying $\A$ once more gives
$(\A)^3(\Psi) \times \A(\Psi_2)  = 0$, i.e.,
$W \times (\A)^2(W) = 0$, so that $W$ is complex-coassociative, i.e., positive.

By Lemma \ref{le:expl-W}(b),  $\beta^a \cap \varphi^{\perp} = \beta$ when $\beta$ is a positive and of rank $2$, thus,
\emph{in both cases (a) and (b), $W = \beta^a \cap \varphi^{\perp}$
provides a positive maximally isotropic subbundle of\/ $\varphi^{\perp}$ containing $\beta$ and satisfying \eqref{W-cond}}.

We set $\l = (\A)^2(W)$.  Since $\A(W)$ and $(\A)^2(W)$ are in
$\varphi$ and $\varphi^{\perp}$, respectively, and so linearly independent,
they span a rank 2 subbundle $D$;
by \eqref{W-cond} this is isotropic.   
The subbundles $\l$ and $D$ are clearly holomorphic and satisfy
$\A(D) \subset \l$; by \eqref{last-leg}
$\l$ is in the kernel of $\A$. 
Further, $W$ is complex-coassociative and satisfies \eqref{W-cond},
Applying $\A$ gives
$\A(W) \times (\A)^2(W) = 0$.   It follows that $D \subset \l^a$.
Since $D \ominus \l \subset \varphi$, $(\l,D)$ is a lift of $\varphi$; by Lemma
\ref{le:J2-descr}, it is $J_2$-holomorphic. The lemma follows.
\end{proof}

\section{Twistor spaces as flags} \label{sec:proofs}

To understand our constructions better, we embed our exceptional
symmetric space in a Grassmannian; we now recall some methods for those.

\subsection{Harmonic maps into real and complex Grassmannians}
\label{subsec:Grass} 

For the case of real or complex Grassmannians, the twistor spaces constructed by \cite{J2} admit descriptions as geometric flag manifolds, as follows.

As before, for any $k,n$ with $0 \leq k \leq n$, let $G_k(\cn^n)$ denote the Grassmannian of $k$-dimensional subspaces of $\cn^n$.
We identify a smooth map $\varphi:M \to G_k(\cn^n)$ with the rank $k$ subbundle 
of $\CC^n = M \times \cn^n$, also denoted by $\varphi$,
whose fibre at a point $p \in M$ is $\varphi(p)$.

For a subbundle $V$ of $\CC^n$ we denote by $\pi_V$
(resp.\ $\pi_V^{\perp}$) orthogonal projection from $\CC^n$ to $V$
(resp.\ to its orthogonal complement $V^{\perp}$).
The Cartan embedding is given by 
\be{cartan}
\iota:G_k(\cn^n)\hookrightarrow \U n,\quad \iota(V)=\pi_V-\pi_{V}^{\perp}\,;
\ee
this is totally geodesic, and isometric up to a constant factor.
We shall identify $V$ with its image $\iota(V)$; since $\iota(V^{\perp}) = -\iota(V)$,
this identifies $V^{\perp}$ with $-V$.
We consider the real Grassmannian $G_k(\rn^n)$ to be the totally
geodesic submanifold of $G_k(\cn^n)$ given by $\{ V \in G_k(\cn^n) : V = \ov{V} \}$.
 
We now recall some methods for studying maps into Grassmannians
\cite{burstall-wood}.
Any subbundle $\varphi$ of $\CC^n$ inherits a metric by restriction, and a connection $\nabla_{\!\varphi}$ by orthogonal projection of the flat connection on $\CC^n$, i.e.,
$(\nabla_{\!\varphi})_Z(v) = \pi_{\varphi}(\pa_Z v)$ \ $(Z \in \Gamma(TM), \ v \in \Gamma(\varphi)\,)$. \ 
Let $\varphi$ and $\psi$ be two mutually orthogonal subbundles of
$\CC^n$.  By the
\emph{$\pa'$ and $\pa''$-second fundamental forms of $\varphi$ in $\varphi \oplus \psi$} we mean the vector bundle morphisms
$A'_{\varphi,\psi}\,, A''_{\varphi,\psi}:\varphi \to \psi$
defined on each coordinate domain $(U,z)$ by
\be{2nd-f-f}
A'_{\varphi,\psi}(v) = \pi_{\psi}(\pa_z v) \quad \text{and} \quad
A''_{\varphi,\psi}(v) = \pi_{\psi}(\pa_{\zbar}v)
	\qquad (v \in \Gamma(\varphi)\,).
\ee

The next result follows from this definition, the last from Lemma \ref{le:diff1}(ii).

\begin{lemma} \label{le:2ffprops}
For any mutually orthogonal subbundles $\varphi$ and $\psi$ of $\CC^n$,
\begin{itemize}
\item[(i)]  $A''_{\psi,\varphi}$ is minus the adjoint of $A'_{\varphi,\psi};$
\item[(ii)]  $A''_{\ov\varphi,\ov\psi}$ is the conjugate of $A'_{\varphi,\psi};$
\item[(iii)] if $\varphi$ is isotropic and of rank one,
$A'_{\varphi,\ov\varphi}$ and $A''_{\varphi,\ov\varphi}$ vanish. \qed
\end{itemize}
\end{lemma}

In particular, we have the \emph{second fundamental forms of $\varphi$}:
$A'_{\varphi} = A'_{\varphi,\varphi^{\perp}}:\varphi \to \varphi^{\perp}$ and
$A''_{\varphi} = A''_{\varphi,\varphi^{\perp}}:\varphi \to \varphi^{\perp}$;
on identifying $\varphi: M \to G_k(\cn^n)$ with its composition
$\iota\circ\varphi:M \to \U n$ with the Cartan embedding \eqref{cartan},
we see that the fundamental endomorphism
$\A$ of \eqref{type-decomp} is minus the direct sum of $A'_{\varphi}$
and $A'_{\varphi^{\perp}}$, similarly the connection
$D^{\varphi}$ of \S \ref{subsec:Lie-groups} is the direct sum of
$\nabla_{\!\varphi}$ and $\nabla_{\!\varphi^{\perp}}$\,.
Hence, writing $\nabla''_{\varphi} = (\nabla_{\!\varphi})_{\pa/\pa\zbar}$\,,
\emph{a smooth map $\varphi:M \to G_k(\cn^n)$ is
harmonic if and only if\/ $A'_{\varphi}$ is holomorphic, i.e.,
$A'_{\varphi} \circ \nabla''_{\varphi}
	= \nabla''_{\varphi^{\perp}} \circ A'_{\varphi}$}\,, 
see also \cite[Lemma 1.3]{burstall-wood}.

We fill out zeros as in \S \ref{subsec:Lie-groups} to obtain subbundles
$G'(\varphi) = \Ima A' _{\varphi} = \Ima(\A|_{\varphi})$ and
$G''(\varphi) = \Ima A''_{\varphi} = \Ima(A^{\varphi}_{\zbar}|_{\varphi})$,
called the \emph{$\pa'$- and $\pa''$-Gauss transforms} or
\emph{Gauss bundles of $\varphi$}; these define maps into Grassmannians which are also harmonic, see \cite[Proposition 2.3]{burstall-wood}, \cite{wood60}; in fact, these operations are examples of \emph{adding a uniton} in the sense of
Uhlenbeck \cite{uhlenbeck}.

More generally, we define the \emph{$i$th $\pa'$-Gauss transform}
 $G^{(i)}(\varphi)$ of a harmonic map $\varphi:M \to G_k(\cn^n)$
by $G^{(0)}(\varphi) = \varphi$,  $G^{(i)}(\varphi) = G'(G^{(i-1)}(\varphi))$, and the \emph{$i$th $\pa''$-Gauss transform} $G^{(-i)}(\varphi)$ by
 $G^{(-i)}(\varphi) = G''(G^{(-i+1)}(\varphi))$.  Note that
$G^{(1)}(\varphi) = G'(\varphi)$ and $G^{(-1)}(\varphi) = G''(\varphi)$.
The sequence $(G^{(i)}(\varphi))_{i \in \zn}$ of harmonic maps is called
\cite{wolfson} the \emph{harmonic sequence of\/ $\varphi$}.  
By Lemma \ref{le:2ffprops}, if
$\varphi$ is a harmonic map into a \emph{real} Grassmannian, we have
$A''_{\varphi} = \ov{A'_{\varphi}}$ and
$G^{(-i)}(\varphi) = \ov{G^{(i)}(\varphi)}$ $\forall\, i \in \zn$.

A harmonic map $\varphi:M \to G_k(\cn^n)$ is called
\emph{strongly conformal} \cite{burstall-wood} if $G'(\varphi)$ and
$G''(\varphi)$ are orthogonal, equivalently $s(\varphi) = 1$.  For a
harmonic map $\varphi:M \to G_k(\rn^n)$ into a \emph{real} Grassmannian, strong conformality is equivalent to isotropy of the subbundle $G'(\varphi)$.

\subsection{Twistor lifts of maps into complex Grassmannians}
\label{subsec:twistor-Grass}

For any complex vector spaces or vector bundles $E$ and $F$,
$\Hom(E,F)$ will denote the vector space or bundle of complex-linear maps from $E$ to $F$.

Let $n,t,d_0,d_1, \ldots, d_t$ be positive integers with $\sum_{i=0}^t d_i = n$. Let $F=F_{d_0,\ldots, d_t}$ be the (geometric complex) flag manifold
$\U{n}\big/\U{d_0}\times\dots\times\U{d_t}$. The elements of $F$ are $(t+1)$-tuples $\psi = (\psi_0,\psi_1, \ldots, \psi_t)$ of mutually orthogonal subspaces with $\psi_0\oplus\dots\oplus \psi_t=\cn^n$; we call these subspaces the \emph{legs of\/ $\psi$}. There is a canonical embedding of $F$ into the product $G_{d_0}(\cn^n)\times\dots\times G_{d_0+\dots+d_{t-1}}(\cn^n)$ given by sending 
$(\psi_0,\psi_1, \ldots,\psi_t)$ to its \emph{associated flag} $(T_0,\dots,T_{t-1})$ where $T_i=\sum_{j=0}^i\psi_j$\,;
the restriction to $F$ of the K\"ahler structure on this product is an (integrable) complex structure
which we denote by $J_1$.   Then the complexified tangent space is
$T^{\cc}_{J_1}F = T^{1,0}_{J_1}F \oplus T^{0,1}_{J_1}F$ with $(1,0)$- and $(0,1)$-spaces at $\psi$:
\be{holo-tgt-space}
T^{1,0}_{J_1}F = \sum_{\substack{i,j=0,\dots,t \\ i<j}}\Hom(\psi_i,\psi_j)\,, \qquad
T^{0,1}_{J_1}F = \sum_{\substack{i,j=0,\dots,t \\ j<i}}\Hom(\psi_i,\psi_j)\,.
\ee
Set $k = \sum_{j=0}^{[t/2]}d_{2j}$ and $N = G_k(\cn^n)$. We define
a map which combines the even-numbered legs:
\be{twistor-proj}
\pi_e: F_{d_0,\ldots, d_t} \to G_k(\cn^n),\qquad \psi  = (\psi_0,\psi_1, \ldots,\psi_t)\mapsto\sum_{j=0}^{[t/2]}\psi_{2j};
\ee
we can also define a projection $\pi_{\text{\it odd}}$ which combines odd-numbered legs, so that $\pi_{\text{\it odd}}(\psi)$ is the orthogonal complement of $\pi_e(\psi)$.
The projection $\pi_e$ is a Riemannian submersion with respect to the standard homogeneous metrics on $F$ and $G_k(\cn^n)$ (given by the restrictions of minus the Killing form of $\U{n}$), so that each tangent space decomposes into the direct sum of the \emph{vertical space} (tangents to the fibres), and its orthogonal complement, the \emph{horizontal space}.

We define an almost complex structure $J_2$ by changing the sign of $J_1$ on the vertical space. This gives $(1,0)$-horizontal and vertical spaces as follows:
\be{hor-vert}
\Hh_{J_2}^{1,0}F
= \sum_{\substack{i,j=0,\dots,t \\ i<j,\ j-i \text{ odd}}}\Hom(\psi_i,\psi_j)\,,
\qquad
\Vv_{J_2}^{1,0}F =\sum_{\substack{i,j=0,\dots,t\\ j < i,\ j-i \text{ even}}} \Hom(\psi_i,\psi_j).
\ee
The almost complex structure $J_2$ is not integrable except when $t=1$, see \cite{burstall-salamon}.  However,

\begin{proposition} \label{pr:pi-twistor} \cite{burstall-rawnsley}
$\pi_e:(F,J_2) \to G_k(\cn^n)$ is a twistor fibration for harmonic maps.
\end{proposition}   

Let $\psi = (\psi_0,\psi_1, \ldots, \psi_t): M \to F$ be a smooth map from a Riemann surface.
The last proposition says that, if $\psi$ is $J_2$-holomorphic, then its twistor projection $\pi_e \circ \psi$ is harmonic.
 
In \cite{J2}, the authors developed a general method for producing twistor lifts
of harmonic maps which used the following definition.

\begin{definition} \label{def:Az-filt}
Let $\varphi:M \to \U{n}$ be a smooth map.
 A filtration of $\CC^n$ by subbundles:
\be{Az-filt}
\CC^n =Z_0  \supset  Z_1  \supset  \cdots  \supset  Z_t  \supset  Z_{t+1} = \ul{0}\, 
\ee
is called an \emph{$A_z^{\varphi}$-filtration (of length $t$)} if,
for each $i = 0,1,\ldots, t$,
\begin{enumerate}
 \item[(i)] $Z_i$ is a holomorphic subbundle of $(\CC^n,D_{\zbar}^{\varphi})$, i.e.,
		$\Gamma(Z_i)$ is closed under $D_{\zbar}^{\varphi}$\,;
 \item[(ii)] $A_z^{\varphi}$ maps $Z_i$ into the smaller subbundle $Z_{i+1}$\,. 
\end{enumerate}
\end{definition}
Note that each $Z_i$ is a \emph{uniton} \cite{uhlenbeck} for $\varphi$, i.e., a holomorphic subbundle of $(\CC^n,D^{\varphi}_{\zbar})$ which is closed under $A_z^{\varphi}$, and any uniton can be
embedded in a $\A$-filtration \cite[Theorem 5.8]{J2}.

Clearly, $\A$-filtrations of $\varphi$ exist if and only if $\varphi$ is nilconformal.  By the \emph{legs} of a filtration \eqref{Az-filt} we mean
the subbundles
$\psi_i = Z_i \ominus Z_{i+1}$ \ $(i= 0,1,\ldots, t)$.
We say that a filtration \eqref{Az-filt} is (i) \emph{strict} if all
the inclusions $Z_{i+1} \subset Z_i$ are strict, equivalently,
all the legs $\psi_i$ are non-zero; 
(ii) \emph{alternating $($for $\varphi$ or $-\varphi)$} if
\begin{equation}\label{alternate}
\psi_i \subset (-1)^i\varphi \quad \forall\, i = 0,1,\ldots, t
\quad \quad \text{or} \quad\quad
\psi_i \subset (-1)^{i+1}\varphi 
\quad \forall\, i = 0,1,\ldots, t.
\ee

Let $\varphi:M \to \U{n}$ be nilconformal of nilorder $r$.  Then two examples of
$\A$-filtrations are the \emph{filtration by $\A$-images}
given by $Z_i = \Ima (\A)^i$ and its `dual', the \emph{filtration by $\A$-kernels}:
$Z_i = \ker (\A)^{r-i}$.
By nilpotency, these filtrations are strict; they coincide when $r=n$, in which case their legs are all of rank one.
The utility of $\A$-filtrations is shown by the following result \cite[Proposition 3.5]{J2}. 

\begin{proposition} \label{lifts-filts}
Let $\varphi:M \to G_k(\cn^n)$ be a nilconformal harmonic map from a Riemann surface.
Then setting $\psi_i = Z_i \ominus Z_{i+1}$ defines a
one-to-one correspondence between $J_2$-holomorphic lifts
$\psi = (\psi_0, \psi_1, \ldots, \psi_t)$ of $\varphi$ into a complex
flag manifold $F_{d_0,\ldots, d_t}$ and
strict alternating $\A$-filtrations \eqref{Az-filt} of length $t$.
\end{proposition}

\subsection{Twistor lifts of maps into real Grassmannians}
\label{subsec:twistor-real}
Let $G_k(\rn^n)$ denote the Grassmannian of real $k$-dimensional subspaces of
$\rn^n$, and $\wt G_k(\rn^n)$ the Grassmannian of \emph{oriented} real
$k$-dimensional subspaces of $\rn^n$; let $\pp:\wt G_k(\rn^n) \to G_k(\rn^n)$
be the double cover which forgets orientation.
Note that $\wt G_k(\rn^n)$ is inner if and only if $k$ or $n-k$ is even
\cite[p.\ 38]{burstall-rawnsley}.
Since taking the orthogonal complement defines isometries from
$G_k(\rn^n)$ to $G_{n-k}(\rn^n)$ and from $\wt G_k(\rn^n)$ to
$\wt G_{n-k}(\rn^n)$, we shall assume without loss of
generality that $n-k$ is even.
To obtain twistor spaces and harmonic maps into these Grassmannians, we
start with a filtration \eqref{Az-filt} which is
\emph{real} in the sense that $Z_i = \ov{Z}^{\perp}_{t+1-i}$, equivalently,
its legs satisfy $\psi_i = \ov{\psi}_{t-i}$ \ $(i=0,1,\dots, t)$.   
Then, if the filtration is strict and alternating and the length $t$ of the filtration is odd, it has an even number of legs, and $\psi$ projects under
$\pi_e$ (or $\pi_{\text{\it odd}}$) to a map into the symmetric space
$\SO{2m}/\U{m}$, see \cite[\S 6.4]{J2}.

We are more interested in the case when $t$ is even.
In that case, it is convenient to set $t=2s$ and renumber the legs:
$\psi = (\psi_{-s}, \ldots, \psi_0, \ldots, \psi_s)$ so that these satisfy
the reality condition:
\be{reality}
\psi_{-i} = \ov{\psi}_i  \qquad (i = 1,\ldots, s) \,;
\ee
thus all the legs are determined by $\psi_i$ with $i$ positive, and the middle leg
$\psi_0 = \CC^n \ominus \sum_{i=1}^s (\psi_i + \ov{\psi}_i)$ is real.
(As usual, we write $\ov{\psi}_i$ for $\ov{\psi_i}$\,, etc.)

This suggests that, for any natural numbers $d_0,\ldots,d_s$ with $d_0 + 2(d_1+\cdots + d_s) = n$, we define submanifolds of the complex flag manifolds in the last subsection by
\be{real-flag}
F^{\rn}_{d_s,\ldots, d_0}
= \bigl\{\psi = (\psi_{-s},\ldots, \psi_0, \ldots, \psi_s) \in F_{d_{-s},\ldots,d_0, \ldots, d_s} : \psi_i = \ov\psi_{-i} \ \forall i \bigr\}\,.
\ee
Here we set $d_{-i} = d_i$ and the notation is consistent with that in \cite[\S 6]{J2}.
As a homogeneous space, $F^{\rn}_{d_s,\ldots, d_0} = \Orthog{n}/H$, where
$H = \U{d_s} \times \cdots \times \U{d_1} \times \Orthog{d_0}$; equally well, $F^{\rn}_{d_s,\ldots, d_0} = \SO{n}/\wt H$ where
$\wt H = \U{d_s} \times \cdots \times \U{d_1} \times \SO{d_0}$.
The twistor projection $\pi_e$ or $\pi_{\text{\it odd}}$
of \eqref{twistor-proj} restricts to give twistor fibrations:
$\pi_0^{\rn}:F^{\rn}_{d_s,\ldots, d_0} \to \Orthog{n} \bigl/ \Orthog{k} \times \Orthog{n-k} = G_k(\rn^n)$  and
$\wt\pi_0^{\rn}: F^{\rn}_{d_s,\ldots, d_0}
= \SO{n}/\wt H  \to \SO{n} \bigl/ \SO{k} \times \SO{n-k} = \wt G_k(\rn^n)$
where $n-k = 2\sum_{i \in \nn,\ i \text{ odd}} d_i$; more geometrically, 
\[
\pi_0^{\rn}(\psi) = \wt\pi_0^{\rn}(\psi) = 
\sum_{i \in \zn,\ i \text{ even}}\psi_{i}\,,
\]
with $\wt\pi_0^{\rn}(\psi)$ given the induced orientation.
These provide twistor spaces for $G_k(\rn^n)$ and $\wt G_k(\rn^n)$ when $n-k$ is even.

For a smooth map $\varphi:M \to \wt G_k(\rn^n)$ from a Riemann surface, we define $\A$ by forgetting the orientation of $\varphi$, i.e., $\A = A^{\pp \circ \varphi}_z$
where $\pp:\wt G_k(\rn^n) \to G_k(\rn^n)$ is the double cover.
We have the following version of Proposition \ref{lifts-filts} adapted to the
real case.

\begin{proposition}
Let $\varphi:M \to \wt G_k(\rn^n)$ be a nilconformal harmonic map where
$n-k$ is even.
Then setting  $\psi_i = Z_{i+s} \ominus Z_{i+s+1}$ defines a
one-to-one correspondence between $J_2$-holomorphic lifts
$\psi = (\psi_{-s}, \ldots, \psi_0, \ldots, \psi_s)$ of\/ $\varphi$ into
a flag manifold \eqref{real-flag} and
\emph{real} strict alternating $\A$-filtrations \eqref{Az-filt} of length $2s$.
\qed
\end{proposition}

\begin{proposition} \label{pr:real-Grass}
Let $\psi:M \to F^{\rn}_{d_s,\ldots,d_0}$ be a $J_2$-holomorphic map into a flag manifold \eqref{real-flag} with
non-constant twistor projection 
$\varphi:=\wt\pi_0^{\rn} \circ \psi:M \to  \wt G_k(\rn^n)$.
Then $\varphi$ is a nilconformal harmonic map into an oriented real Grassmannian with $n-k$ even.   Further, $s(\varphi) \leq s$.
\end{proposition}

\begin{proof}
That $\varphi$ is harmonic follows from Proposition \ref{pr:pi-twistor}.
Since $n-k = 2\sum_{i \in \nn,\ i \text{ odd}} d_i$\,, it is even.
Since $(-1)^{s-1}\varphi = \sum_{i=0}^{s-1}\psi_{-s+1+2i}$ and $(\A)^2$ maps
$\psi_{-s+1+2i}$ to $\sum_{j > i} \psi_{-s+1+2j}$, we see that 
$(\A)^{2s}\bigl((-1)^{s-1}\varphi\bigr) = 0$ so that $\varphi$ is nilconformal
with $s(\varphi) \leq s$.
\end{proof}

The inequality $s(\varphi) \leq s$ cannot be replaced by equality, see
Remark \ref{re:two-lifts}.

\subsection{Twistor spaces of $G_2/\SO{4}$ as flags} \label{subsec:twistor-excep} 

We saw in \S \ref{subsec:octonians} that $G_2$ has three flag manifolds which fibre over $G_2/\SO{4}$; these give three twistor spaces for $G_2/\SO{4}$.
To study these twistor spaces geometrically, we embed
$G_2/\SO{4}$ in the oriented Grassmannian $\wt G_3(\rn^7)$. The twistor spaces of this Grassmannian are flag manifolds $F = F_{d_s\ldots,d_0}^{\rn}$ as described in the last subsection.   Since $d_0 + 2(d_1 +\cdots + d_s) = 7$ and 
$\sum_{i>0, i \text{ odd}}d_i = 2$ there are precisely three possibilities which we will denote by $F_1$, $F_2$, $F_3$:
\[
{\rm (i)} \ s=1 \text{ and } F_1 = F_{2,3}^{\rn}\,; \quad
{\rm (ii)} \ s=2 \text{ and } F_2 = F_{1,2,1}^{\rn}\,; \quad
{\rm (iii)} \ s=3 \text{ and } F_3 = F_{1,1,1,1}^{\rn}\,.
\]

\begin{definition} \label{def:G2-flag} Let
$\psi = (\psi_{-s},\ldots,\psi_s) \in F_s$ \ $(s=1,2$ or $3$).
We shall say that $\psi$ is a \emph{$G_2$-flag} if
\be{G2-flag}
\psi_i \times \psi_j \subset \psi_{i+j} \quad \forall i,j \in \zn.
\ee
Here, we set $\psi_i = 0$ for $|i| > s$.  
We denote the set of all $G_2$-flags in $F_s$ by $\T_s$\,.
\end{definition}

Recall from \S \ref{subsec:repn} that the annihilator
$\l^a = \{X \in \I\O : X \times \l = 0\}$ of a $1$-dimensional
isotropic subspace $\l$ of $\I\O$ is isotropic,
of dimension three, and contains $\l$.
The following should be compared with \cite[\S 5.2]{kobak}.

\begin{proposition}\label{pr:Z-char} Let $\psi = (\psi_{-s},\ldots,\psi_s) \in F_s$ for some $s \in \{1,2,3\}$.
The following are equivalent$:$

{\rm (i)} $\psi$ is a $G_2$-flag (Definition \ref{def:G2-flag}),
i.e., $\psi \in \T_s;$

{\rm (ii)} the $\psi_i$ are the legs of a filtration$:$
\be{F-filts}
\left. \begin{array}{rl}
(s=1) \quad &\cn^7 \supset \ov W^{\perp} \supset W \supset 0\,,  \\
(s=2) \quad &\cn^7 \supset \ov \l^{\perp} \supset \ov{(\l^a)}^\perp \supset \l^a \supset \l \supset 0 \,, \\
(s=3) \quad &\cn^7 \supset \ov \l^{\perp} \supset \ov D^{\perp} \supset
\ov{(\l^a)}^\perp \supset \l^a \supset D \supset \l \supset 0\,,
\end{array} \right\}
\ee
where $W$ is a $2$-dimensional complex-coassociative subspace of\/ $\cn^7$, $\l$ is a
$1$-dimensional isotropic subspace of\/ $\cn^7$, and
$D$ is a $2$-dimensional subspace of\/ $\l^a$ which contains $\l;$

{\rm (iii)} $\psi$ is $G_2$-equivalent to the flag$:$
\be{std-flag1}
\left. \begin{array}{rl}
(s=1) \ &(\l_{-2\alpha_1-\alpha_2} \oplus \l_{-\alpha_1 - \alpha_2}\,,\,
 \l_{-\alpha_1} \oplus \l_0 \oplus \l_{\alpha_1}\,,\,
	\l_{\alpha_1+\alpha_2} \oplus \l_{2\alpha_1 + \alpha_2})\,,
\\
(s=2) \ &(\l_{-\alpha_1} \,,\, \l_{\alpha_1+\alpha_2} \oplus \l_{-2\alpha_1-\alpha_2}\,,\, \l_0 \,,\,
	\l_{-\alpha_1-\alpha_2} \oplus \l_{2\alpha_1+\alpha_2}\,,\, \l_{\alpha_1})\,,
	\\
(s=3) \ &(\l_{-2\alpha_1-\alpha_2} \,,\, \l_{-\alpha_1} \,,\,
\l_{-\alpha_1-\alpha_2}\,,\, \l_0 \,,\,
	\l_{\alpha_1+\alpha_2} \,,\, \l_{\alpha_1}\,,\, \l_{2\alpha_1 + \alpha_2})\,.
\end{array} \right\}
\ee
\end{proposition}

\begin{proof}
We show that (i) $\implies$ (ii) $\implies$ (iii) $\implies$ (i).

First, suppose that (i) holds.  When $s=1$ put $W = \psi_1$,
when $s=2$ put $\l=\psi_2$, when $s=3$ put $\l = \psi_3$ and $D = \psi_2 \oplus \psi_3$; then we get filtrations as in \eqref{F-filts}, so that (ii) holds.

Next, suppose that (ii) holds.  Then by Lemma \ref{le:G2-trans},
we can find an element of $G_2$
which, when $s=1$, maps  $W$ to the complex-coassociative subspace
$\l_{\alpha_1+\alpha_2} \oplus \l_{2\alpha_1 + \alpha_2}$;
when $s=2$, maps $\l$ to the isotropic line $\l_{\alpha_1}$; and
when $s=3$, maps $\l$ and $D$ to $\l_{2\alpha_1 + \alpha_2}$ and
$\l_{\alpha_1} \oplus  \l_{2\alpha_1 + \alpha_2}$, respectively.
The rest of the legs in \eqref{std-flag1} then follow by calculation.  For example,
when $s=3$, ${(\l_{2\alpha_1+\alpha_2})}^a$ is spanned by
$\l_{\alpha_1}$, $\l_{\alpha_1+\alpha_2}$ and $\l_{2\alpha_1 + \alpha_2}$, determining the legs $\psi_i$ for $i$ positive;
the other legs follow by the reality condition \eqref{reality}.
Hence (iii) holds.

Lastly, suppose that (iii) holds.   Then simple calculations
using \eqref{weight-mult} show that the flags \eqref{std-flag1} satisfy
\eqref{G2-flag}, so that (i) holds.
\end{proof}

Now note that the projections $\pi_0:F_s \to \wt G_3(\rn^7)$ which combine even-numbered legs restrict to projections $\pi_0:\T_s \to G_2/\SO{4}$.

\begin{remark}  We have chosen flags in \eqref{std-flag1} which all project to
$\varphi = \l_{-\alpha_1} \oplus \l_0 \oplus \l_{\alpha_1}$ and so give a model for the commutative diagram \eqref{diag:twistor-spaces}; if we drop this requirement, we can apply elements of the Weyl group to produce alternative flags (cf.\ \cite{kobak}) which are $G_2$-equivalent to those of \eqref{std-flag1}.  In particular, the following flags will be useful to us in \S \ref{subsec:twistor-fibr2}; the roots appearing in the same order as \eqref{std-flag1}($s$=1).
\be{std-flag2}
\left. \begin{array}{rl}
(s=2) \ &(\l_{-2\alpha_1-\alpha_2} \,,\, \l_{-\alpha_1-\alpha_2} \oplus \l_{-\alpha_1}\,,\, \l_0 \,,\,
	\l_{\alpha_1} \oplus \l_{\alpha_1+\alpha_2}\,,\, \l_{2\alpha_1+\alpha_2})\,,
	\\
(s=3) \ &(\l_{-2\alpha_1-\alpha_2} \,,\, \l_{-\alpha_1-\alpha_2} \,,\,
\l_{-\alpha_1}\,,\, \l_0 \,,\,
	\l_{\alpha_1} \,,\, \l_{\alpha_1+\alpha_2}\,,\, \l_{2\alpha_1 + \alpha_2})\,.
\end{array} \right\}
\ee
\end{remark}

We now show that the $G_2$-fibre bundles $\pi_0:\T_s \to G_2/\SO{4}$ are isomorphic to the three twistor bundles $T_s \to G_2/\SO{4}$ defined in \S \ref{subsec:octonians}.   To do this, we need the following result
which will be proved by Lie theory in \S \ref{subsec:twistor-fibr2}.  

\begin{lemma} \label{le:preserve-J2}
The following maps intertwine both $J_1$ and $J_2:$
\be{inclusions}
\left. \begin{array}{rcl}
 T_1 &=& G_2/\U{2}_+ \to F_1 \quad \text{given by $W \mapsto \psi  =$ the legs of
\eqref{F-filts}\emph{($s$=1);}}\\
 T_2 &=& Q^5 \to F_2 \quad \text{given by $\l \mapsto \psi = $ the legs of \eqref{F-filts}\emph{($s$=2);}}\\
T_3 &=& G_2/(\U{1} \times \U{1}) \to F_3 \quad \text{given by $(\l,D) \mapsto \psi$}\\
    & & \qquad \qquad \qquad \qquad\qquad \qquad \text{$= $ the legs of \eqref{F-filts}\emph{($s$=3)}}.
\end{array} \right\}
\ee
\end{lemma}

In fact, these maps are injective with image $\T_s$\,, as shown by the following result.

\begin{proposition} \label{pr:T-and-T}
The almost complex structures $J_1$ and $J_2$ on $F_s$ restrict to ones on $\T_s$ which we shall also denote by $J_1$ and $J_2$.  
The maps \eqref{inclusions} define $G_2$-equivariant isomorphisms of fibre bundles $T_s \to \T_s$
over $G_2/\SO{4}$ which intertwine both $J_1$ and $J_2$.
\end{proposition}

\begin{proof}
By definition, the image of each map is contained in $\T_s$.
It is easy to check that the following are the inverse mappings: 

$\T_1 \to T_1 = G_2/\U{2}_+$ given by $\psi \mapsto W$ where $W = \psi_1\,;$

$\T_2 \to T_2 = Q^5$ given by $\psi \mapsto \l$ where $\l = \psi_2\,;$ 

$\T_3 \to T_3 = G_2/(\U{1} \times \U{1})$ given by $\psi \mapsto$ the pair $(\l, D)$ where $\l = \psi_2$ and $D= \psi_2 \oplus \psi_3$.

Further, the maps $T_s \to \T_s$ intertwine the projections $\pi_s:T_s \to G_2/\SO{4}$
of \eqref{diag:twistor-spaces}
and $\pi_0:\T_s \to G_2/\SO{4}$ defined above, and are $G_2$-equivariant, so are isomorphisms of fibre bundles.
We may transfer the almost complex structures $J_1$ and $J_2$ from $T_s$ to $\T_s$ by the isomorphisms;  by Lemma \ref{le:preserve-J2}, these agree with those obtained by restriction from $F_s$. 
\end{proof}

Note also that we have a $G_2$-equivariant isomorphism of fibre bundles: $G_2/\U{2}_- \to \T_2$ 
defined by $W \mapsto \psi = $ the legs of \eqref{F-filts}($s$=2) with $\l = W \times W$; this has inverse $\psi \mapsto W = \psi_1$.  This commutes with the isomorphism of $Q^5$ to $G_2/\U{2}_-$ given in \S \ref{subsec:twistor-sp1}.

\subsection{Uniqueness of the twistor lifts}
Using the identifications of our twistor spaces with the spaces of $G_2$-flags
$\T_s$ as explained in the previous section, we can establish the uniqueness part of Theorem \ref{th:lifts}.  

\begin{lemma} \label{le:uniqueness}
Let $\varphi:M \to G_2/\SO{4}$ be a nilconformal harmonic map and set $s = s(\varphi)$.
Then any two $J_2$-holomorphic lifts into $T_s$ are equal.
\end{lemma}

Note, however, that Remark \ref{re:two-lifts} shows that a `sufficiently simple'
nilconformal harmonic map may also have a lift into a twistor space $T_s$ with
$s > s(\varphi)$ (which might not be unique).

\begin{proof}
$(s=1)$ \ Let $\psi:M \to T_1 = G_2/\U{2}_+$ be a $J_2$-holomorphic lift of $\varphi$.  Then $\psi_1 = W =$ the unique
 complex-coassociative subbundle of $\varphi^{\perp}$ containing
$G'(\varphi)$.

$(s=2)$ \ Let $\psi:M \to T_2 = Q_5 \cong \T_2$ be a
$J_2$-holomorphic lift of $\varphi$.

Note that $A'_{\psi_{-2},\psi_{-1}} \neq 0$ otherwise $s(\varphi) = 1$.
Further, either $A'_{\psi_0} = 0$ in which case
$\ker A'_{\varphi} = \psi_0 + \psi_2$, or
$A'_{\psi_0} \neq 0$ in which case $\ker A'_{\varphi} = \psi_2$. 
In both cases $\psi_2 = \l$ is characterized as the unique
isotropic line subbundle in $\ker A'_{\varphi}$.

$(s=3)$ \ Let $\psi:M \to T_3$ be a $J_2$-holomorphic lift of $\varphi$.
Set $\beta = (\A)^3(\varphi)$.  Then $\beta \neq 0$, otherwise we would have
$s(\varphi) \leq 2$.  Further, $\beta \subset \psi_1 \oplus \psi_3$.
Hence $\psi_1 \oplus \psi_3$ is the unique complex-coassociative
subbundle of $\varphi^{\perp}$ containing
$\beta$. Next,  $(\A)^2(\psi_1 \oplus \psi_3)$ is contained in $\psi_3$,
and is is zero if and only if $A'_{\psi_2,\psi_3} \circ A'_{\psi_1,\psi_2} =0$. Since this implies that $s(\varphi) \leq 2$, $(\A)^2(\psi_1 \oplus \psi_3)$ is non-zero and so equals $\psi_3 = \l$; this also fixes $\psi_1 = (\psi_1 \oplus \psi_3) \ominus \psi_3$.  Lastly
$D = \psi_2 \oplus \psi_3$ is characterized by $D = \l^a \ominus \psi_1$.
\end{proof} 

\subsection{Inclusive maps} \label{subsec:inclusive} 
Recall \cite{eells-salamon} that a weakly conformal map from a Riemann surface to a quaternionic K\"ahler manifold is called \emph{inclusive} if the image of each tangent space is contained in a $4$-dimensional quaternionic subspace.
It follows from \cite[Proposition 7.3]{J2} that \emph{$\varphi:M \to G_2/\SO{4}$ is
inclusive if and only if it is strongly conformal}; in that case, our twistor lift $W: M \to T_1$ is the $J_2$-holomorphic lift constructed by J.\ Eells and S.~M.\ Salamon \cite{eells-salamon}.

\subsection{Superhorizontal maps}  \label{subsec:superhor}
The concept of superhorizontal maps into a flag manifold is defined in \cite{burstall-rawnsley}. For maps into a geometric flag manifold it amounts to the following:
\begin{definition} \label{def:superhor} \cite{burstall-rawnsley}
A smooth map $\psi=(\psi_i)$ into a complex flag manifold is called
\emph{superhorizontal holomorphic} if the second
fundamental forms $A'_{\psi_i,\psi_j}$ are zero for all $i,j$ with $j \neq i+1$.
Equivalently, for each $i = 0,1,\ldots$,
$\delta_i = \sum_{j=0}^i \psi_j$ is a holomorphic subbundle of
$(\CC^n, \pa_{\zbar})$ and $\pa_z$ maps sections of $\delta_i$
into sections of $\delta_{i+1}$\,.
\end{definition}

By \eqref{holo-tgt-space} and \eqref{hor-vert}, we see that a superhorizontal holomorphic map is
both $J_1$-holomorphic and $J_2$-holomorphic with respect to $\pi_e$, and so horizontal.
The twistor projection of a superhorizontal holomorphic map is a special type of harmonic
map, see \S \ref{subsec:superhor2}, also \cite[\S 4C]{burstall-rawnsley}, for more information, In the case of maps into a sphere or projective space, the twistor projection of a superhorizontal holomorphic map is called
\emph{superminimal}, see \S 5.2ff.

The superhorizontal holomorphic maps $\psi:M \to \T_s$ are given by the legs of the filtration
\eqref{F-filts}, with $\ov\l$ and $\ov D$ now being holomorphic subbundles of $(\CC^7,\pa_{\zbar})$.

\section{Relationship with almost complex maps to the $6$-sphere}
\label{sec:almost-complex}

\subsection{The almost complex structure on $S^6$}
Let $S^6$ be the unit sphere in $\I\O \cong \rn^7$.   Define an almost complex structure
$J$ on $S^6$ by $J_F v = F \times v$ \ $(F \in S^6, \ 
v \in T_F S^6 = \spa\{F\}^\perp)$.  The following formulae are a consequence of \eqref{vp-scalar}
and \eqref{vp-assoc}.

\begin{lemma} \label{le:alg-S6} Let $F \in S^6$ and let $\alpha,\beta\in T_F^{1,0}S^6$. Then 
\begin{enumerate}
\item[(i)] $\alpha\times\ov\beta=\ii\,(\alpha,\ov \beta)F;$
\item[(ii)] $\alpha \times \beta \in T_F^{0,1}S^6;$
\item[(iii)] $|\alpha\times\beta|^2
=2\bigl(|\alpha|^2|\beta|^2-|(\alpha,\ov\beta)|^2\bigr)$. \qed
\end{enumerate}
\qed \end{lemma}

We can improve Lemma \ref{le:alg-S6}(ii) by using Lie theory, as follows. Fix a Cartan subalgebra $\h$ of $\g_2$. By transitivity of $G_2$ on $S^6$ we can take $F$ in the zero weight space $\l_0$.
Consider now the copy of $\su{3}\subset\g_2$ given by the long roots:
\[
\su{3}^\cn=\h^\cn\oplus \g_{\pm\alpha_2} \oplus  \g_{\pm(3\alpha_1+\alpha_2)} \oplus  \g_{\pm(3\alpha_1+2\alpha_2)}.
\]
Under the action of $\su{3}$, $\cn^7$ decomposes as 
\be{S6-decomp}
\cn^7=(\l_{-\alpha_1}\oplus \l_{-\alpha_1-\alpha_2}\oplus \l_{2\alpha_1+\alpha_2})\oplus \l_0\oplus (\l_{\alpha_1}\oplus \l_{\alpha_1+\alpha_2}\oplus \l_{-2\alpha_1-\alpha_2}).
\ee
Since $\su{3}$ acts trivially on $\l_0$, it will also preserve the $(1,0)$- and
$(0,1)$-spaces of the induced orthogonal complex structure on $\l_0^\perp$.  Those spaces are thus the two isotropic spaces in brackets in \eqref{S6-decomp}; they correspond to the standard representation of $\su{3}$ and its dual. Hence, we can take 
$
T^{1,0}_F S^6 =\l_{\alpha_1}\oplus \l_{\alpha_1+\alpha_2}\oplus \l_{-2\alpha_1-\alpha_2}\,;
$
then a brief calculation shows that $T^{1,0}_F S^6 \times T^{1,0}_F S^6$
\emph{is equal to} $T^{0,1}_F S^6$. 

\subsection{Almost complex maps into $S^6$}
For any $n \in \nn$, let $S^n$ denote the unit sphere in $\rn^{n+1}$.
Given a smooth map $F:M \to S^n$ from a Riemann surface, let $f:M \to \RP{n}$ be its composition with
the standard double covering $S^n \to \RP{n}$; equivalently, $f = \spa\{F\}$ is the real line subbundle of $\CC^{n+1}$ of which $F$ is a section.
Embed $\RP{n}$ in $\CP{n}$.  Then if $F$, equivalently $f$, is harmonic, we
can define the Gauss bundles
$G^{(i)}(f)$ \ $(i \in \zn)$ as subbundles of the trivial bundle
$\CC^{n+1} = M \times \cn^{n+1}$ or, equivalently, maps $M \to \CP{n}$;
these are harmonic and $G^{(-i)}(f) = \ov{G^{(i)}(f)}$ \ $(i \in \zn)$,
cf.\ \S \ref{subsec:Grass}.

The \emph{isotropy order} of $f$ (or $F$) is the maximum $r$ such that
$f \perp G^{(i)}(f)$ \ $(i=1,2,\ldots,r)$, or equivalently
(cf.\ \cite[Lemma 3.1]{burstall-wood}), such that
any $r+1$ consecutive Gauss bundles $G^{(i)}(f)$ are zero or orthogonal
(this condition is called \emph{$(r+1)$-orthogonality} in \cite{bolton-pedit-woodward}) and elsewhere.  
The isotropy order of a harmonic map into $S^n$ or $\RP{n}$ is always odd (see, for example,
\cite{bolton-pedit-woodward}).

A harmonic map $F:M \to S^n$ is called \emph{superminimal}
(or \emph{$($real$)$ isotropic}, or \emph{pseudoholomorphic}
\cite{eells-wood}) if it has infinite isotropy order, i.e.,
$G^{(i)}(f) \perp G^{(j)}(f)$ for all $i,j \in \zn$; this happens as soon as the
 isotropy order is at least $n$.
Without loss of generality, we can assume that $F$ is \emph{(linearly) full}, i.e, does not have image in any proper subspace of $\rn^{n+1}$, then $n$ is even, say $n=2m$. 
A full superminimal harmonic map from a Riemann surface $M$ to $S^{2m}$ is the projection of a (super-)horizontal  map into $F_{m,1}^{\rn}$; equivalently, $f = G^{(m)}(h)$ for some `totally isotropic' holomorphic map $h:M \to \CP{2m}$ \cite{eells-wood}.

A smooth map $M \to S^6$ is called \emph{almost complex} if it is (almost-)holomorphic with respect to $J$, i.e., its differential intertwines the complex structure on 
$M$ with $J$.   Such maps are weakly conformal and harmonic, see \cite{bolton-vrancken-woodward}. Note that, for maps into $\cn P^{n-1}=G_1(\cn^n)$, the notions of strong and weak conformality coincide; in particular, almost complex maps $M\to S^6$ are nilconformal. 

The next result follows from work of \cite{bolton-vrancken-woodward} and
\cite{bolton-pedit-woodward}.  We sketch a proof as we shall need some details from it.   We use diagrams in the sense of\/ \cite{burstall-wood}, i.e., vertices represent orthogonal subbundles $\psi_i$ whose sum in $\CC^n$, and an arrow from $\psi_i$ to $\psi_j$ represents the second fundamental form $A'_{\psi_i,\psi_j}$, the absence of this arrow indicating that $A'_{\psi_i,\psi_j}$ is known to vanish.

\begin{proposition} \label{pr:tau} Let $F:M \to S^6$ be a non-constant almost complex map.  Write $f = \spa\{F\}$.  Then  \emph{either}
\begin{enumerate}
\item[(i)] $F$ is a weakly conformal map into a totally geodesic
$S^2 = \Pi^3 \cap S^6$ where $\Pi^3$ is an associative $3$-dimensional
subspace of\/ $\rn^7$, \emph{or}

\item[(ii)] there is a $G_2$-flag
\[
\psi = (\psi_{-3},\psi_{-2},\psi_{-1},\psi_0,\psi_1,\psi_2,\psi_3):M \to G_2/(\U{1} \times \U{1})
\]
with $\psi_i = G^{(i)}(f)$ \ $(i = -2,\ldots,2)$ and
we have the following diagram showing the only possible non-zero second fundamental forms $A'_{\psi_i,\psi_j}:$

\be{tau}
\xymatrix{ 
\psi_{-3} \ar[r] & \psi_{-2} \ar[r] & \psi_{-1} \ar[r] &
\psi_0 \ar[r] & \psi_1 \ar[r] & \psi_2 \ar[r] \ar@/^1.5pc/[lllll] &
\psi_3 \ar@/_1.5pc/[lllll]
}
\ee
\end{enumerate}
\end{proposition}

\begin{proof} Define $F_i$ iteratively by
\be{Fi}
	F_0=F, \quad F_i = A'_{G^{(i-1)}(f)}(F_{i-1}), \quad \text{and} \quad F_{-i} = \ov F_i \quad (i =1,2,\ldots),
\ee 
 so that the $F_i$ are sections of $G^{(i)}(f)$. Since $F$ is non-constant,
the sections $F_{\pm 1}$  are not identically zero, however, $F_i$ may be zero when $\vert i \vert$ is sufficiently large.  Since $F$ is weakly conformal, it has isotropy order at least $3$ (recall the isotropy order is always odd), so that
$\{F_{-1}, F, F_1\}$ are mutually orthogonal.   Since $F$ is almost complex, we have
\be{F1}
J F_1 = F \times F_1 = \ii F_1\,.
\ee

(i) Suppose that $F_2$ is identically zero.  Then $\Pi^3 = \spa\{F_{-1},F,F_1\}$ is a constant
$3$-dimensional subspace; by \eqref{F1} and Lemma \ref{le:assoc}, this is associative.
Further, $F:M \to S^6$ is a weakly conformal map into the totally geodesic $S^2$ given by $S^2 = \Pi^3 \cap S^6$.

(ii) Suppose, instead, that $F_2$ is not identically zero. Then, with a prime denoting differentiation  
 with respect to a local complex coordinate, \eqref{F1} gives
$J{F_1}' =  F \times {F_1}' = \ii {F_1}'$.  Since $F' = F_1$, ${F_1}' = F_2 \mod F_1$ and
$JF_2$ is orthogonal to $JF_1$, this yields 
\be{F2}
J F_2 = F \times F_2 = \ii F_2\,.
\ee

{}From \eqref{F1} and \eqref{F2}, $F_1$ and $F_2$ both lie in the isotropic
subspace $T^{1,0}_F S^6$; it follows that $F$ has isotropy order at least $5$.
Next, recalling \eqref{Fi}, set $\Psi_i = F_i$ \ $(i = -2,\ldots, 2)$ and define
\be{Psi3}
\Psi_3 = \Psi_1 \times \Psi_2\,, \qquad \Psi_{-3} = \ov{\Psi}_3.
\ee
Note that, unlike $F_3$, $\Psi_3$ is not in general a section of $G^{(3)}(f)$.
However, by Lemma \ref{le:alg-S6}, $\Psi_3$ is a non-zero element of
$T^{0,1}_F S^6$ so that $\Psi_{-3}$ is a non-zero element of $T^{1,0}_F S^6$.
{}From \eqref{Psi3}, $\Psi_{-3}$ is orthogonal to $\Psi_1$
and $\Psi_2$.   Set $\psi_i = \spa\{\Psi_i\}$ \ $(i=-3,\ldots,3)$ and
$\psi_i = 0$ for $|i| > 3$.
Then the bundles $\psi_i$ are mutually orthogonal,
$T^{1,0}_F S^6 = \psi_1 \oplus \psi_2 \oplus \psi_{-3}$ and 
$T^{0,1}_F S^6 = \psi_{-1} \oplus \psi_{-2} \oplus \psi_3$. 

We show that $(\psi_i)$ is a $G_2$-flag, i.e., $\psi_i \times \psi_j = \psi_{i+j}$ \ $(i,j \in \zn)$. 
Indeed, $\Psi_1 \times \Psi_2 = \Psi_3$ implies that 
$\Psi_1 \times \Psi_3 = \Psi_1 \times (\Psi_1 \times \Psi_2) =0$.  Similarly $\Psi_2 \times \Psi_3 = 0$.   Hence, setting
$\l = \psi_3$, and $D = \psi_3 \oplus \psi_2$, the $\psi_i$ are the legs of the standard $G_2$-flag
\eqref{F-filts}($s=3$).

We show that we get the diagram \eqref{tau}.
First, $\psi_{-2}  \to \cdots \to \psi_2$ is part of the
harmonic sequence of $f$ so there are no other second fundamental forms
between these elements.  Next, $\I A'_{\psi_2} = G'(G^{(2)}(f)) = G^{(3)}(f)$.   Since $f$ has isotropy order at least $5$, this is orthogonal to 
$\psi_{-2} \oplus \cdots \oplus \psi_2$, hence
$G^{(3)}(f) = \I A'_{\psi_2} \subset \psi_{-3} \oplus \psi_3$; taking the conjugate shows that 
$G^{(-3)}(f) = \I A''_{\psi_{-2}} \subset \psi_{-3} \oplus \psi_3$.  By Lemma
\ref{le:2ffprops}(iii), $A'_{\psi_3,\psi_{-3}} = 0$. 
It follows that the only possible non-zero second fundamental forms are those shown in \eqref{tau}.
\end{proof}

\begin{remark}
(i) In \cite[p.\ 147]{bolton-pedit-woodward} there is a precise multiplication table for the $\Psi_i$\,.

(ii) A smooth map $\psi = (\psi_{-3},\psi_{-2},\psi_{-1},\psi_0,\psi_1,\psi_2,\psi_3):M \to G_2/\bigl(\U{1} \times \U{1}\bigr)$ with the only possible second fundamental forms as in \eqref{tau} is called \emph{$\tau$-primitive} in
\cite{bolton-pedit-woodward}.

(iii) In \cite{bolton-vrancken-woodward}, almost complex maps $F:M \to S^6$ are classified into four types as follows:
(I) full superminimal harmonic maps into $S^6$,
(II) full non-superminimal harmonic maps into $S^6$,
(III) full non-superminimal harmonic maps into a totally geodesic $S^5$, 
(IV) weakly conformal maps into $S^2 = \Pi^3 \cap S^6$ for some
 associative $3$-dimensional subspace $\Pi^3$.
For a Type (I) map, we have $\psi_3 = G^{(3)}(f)$ and diagram \eqref{tau} reduces to the harmonic sequence:
\be{ha-seq}
G^{(-3)}(f) \to G^{(-2)}(f) \to G^{(-1)}(f) \to f \to G^{(1)}(f) \to G^{(2)}(f) \to G^{(3)}(f),
\ee
 see also \S \ref{subsec:superminimal}.
For a Type (II) map, all the second fundamental forms in
diagram \eqref{tau} are non-zero, so that $G^{(3)}(f) \neq \psi_3$, however
 $G^{(3)}(f) \oplus G^{(-3)}(f) = \psi_3 \oplus \psi_{-3}$\,.
For a Type (III) map, the harmonic sequence of $f$ lies in a $6$-dimensional subspace $Y$ of $\rn^7$ and is periodic of period $6$, hence $G^{(3)}(f) = G^{(-3)}(f)$ and $\psi_3 \oplus \psi_{-3} = G^{(3)}(f) \oplus Y^{\perp}$. 
\end{remark}

\subsection{Building harmonic maps into $G_2/\SO{4}$} \label{subsec:building}

We give a way of building harmonic maps into $G_2/\SO{4}$ from almost complex maps into $S^6$.   We need the following \emph{Reduction Theorem}.

\begin{theorem} \label{th:reduction} \cite[Theorem 4.1]{burstall-wood}
Let $f$ be a harmonic map from a Riemann surface to a complex Grassmannian, and let
$\alpha$ be {\rm (i)} a holomorphic line subbundle of\/ $\ker A'_{f^{\perp}}$, or
{\rm (ii)} an antiholomorphic line subbundle of\/ $\ker A''_{f^{\perp}}$. Set
$\wh{\varphi} = f \oplus \alpha$.  Then

{\rm (a)} $\wh{\varphi}$ is also harmonic$;$

{\rm (b)} $\wh{\varphi}$ is nilconformal if and only if $f$ is.
\end{theorem}

\begin{proof}
(a) This is another example (cf.\ \S \ref{subsec:Grass})
of adding a uniton \cite{uhlenbeck}, or see \cite{burstall-wood}.

(b) It is easily checked that
 $(\A)^{2k+1}(\wh{\varphi}^{\perp}) = (A^f_z)^{2k+1}(f^{\perp}) \ \forall k \in \nn$.  The result follows.
\end{proof}

\begin{proposition} \label{pr:alpha-constr}
Let $F:M \to S^6$ be almost complex and let $\alpha$ be a holomorphic line subbundle of
$F^{-1}T^{1,0}S^6$.   Set
\be{alpha-constr}
\varphi = \ov{\alpha} \oplus \spa\{F\} \oplus \alpha.
\ee
 Then
$\varphi$ is a nilconformal harmonic map from $M$ to $G_2/\SO{4}$.
\end{proposition}

\begin{proof}
Let $L$ be a nowhere zero (local) section of $\alpha$.
By Lemma \ref{le:alg-S6},
$L \times \ov{L}\big/|L|^2 = \ii F$, so by Lemma \ref{le:assoc},
$\varphi$ is associative.  {}From the proof of Proposition \ref{pr:tau},
$F^{-1}T^{1,0}S^6 = \psi_{-3} \oplus \psi_1 \oplus \psi_2 \subset
(\psi_{-1} \oplus \psi_0)^{\perp} \subset \ker A'_{f^{\perp}}$.
Thus $\alpha \subset \ker A'_{f^{\perp}}$ and, by 
Theorem \ref{th:reduction} part (i),
$\wh{\varphi} = f \oplus \alpha$ is a harmonic map into a Grassmannian.

Now $\ov{\alpha}$ lies in $\ker A''_{f^{\perp}}$.
Further, by Lemma \ref{le:2ffprops}(iii), $A''_{\ov{\alpha}, \alpha} = 0$,
so $\ov{\alpha}$ lies in $\ker A''_{{\wh{\varphi}^{\perp}}}$.
Hence, we may apply Theorem \ref{th:reduction} part (ii) to see that
$\varphi = \wh{\varphi} \oplus \ov{\alpha}$ is harmonic.
\end{proof}

\begin{remark} \label{re:fi-uniton-no}
Since $\varphi$ is obtained from $f$ by adding a couple of unitons,
the harmonic map \eqref{alpha-constr} is of finite uniton number
(see \S \ref{sec:finite-uniton}) if and only the
almost complex map $F:M \to S^6$ is of finite uniton number.
See Example \ref{ex:s2} for a specific example.
\end{remark}

\begin{example} \label{ex:F-const}
Suppose that $F$ is constant.
Then $F^{-1}T^{1,0}S^6 = M \times T_F^{1,0}S^6$ is a constant maximally isotropic subbundle of
$f^{\perp} \!\otimes \cn \cong \CC^6$.  Choose an identification of the vector space $T_F^{1,0}S^6$
with $\cn^3$.   Then, given a non-constant holomorphic map
from $M$ to $\CP{2}$, we get a corresponding line subbundle $\alpha$
of $F^{-1}T^{1,0}S^6$, and Proposition \ref{pr:alpha-constr} gives a harmonic map
$\varphi:M \to G_2/\SO{4}$.

In fact, it can be checked that (a) $\varphi$ and $\varphi^{\perp}$ are strongly conformal, (b) $\rank G'(\varphi) =1$, and (c) $G'(\varphi^{\perp}) \times G''(\varphi^{\perp})$ is a constant subbundle
of $\CC^7$ (namely $f$).  The construction gives a one-to-one correspondence between pairs $(F,\alpha)$ as above and harmonic maps $\varphi$ satisfying these three properties. 
\end{example} 

In the case that $F$ is non-constant, by setting $\alpha = G'(f)$ we obtain
harmonic maps into $G_2/\SO{4}$ which complement those of Example
\ref{ex:F-const}, as follows.
Recall that a smooth map $\varphi$ into a Grassmannian is strongly conformal if and only if $s(\varphi) = 1$.

\begin{theorem} \label{th:one-one} There is a one-to-one correspondence \emph{between} 
\begin{enumerate}
\item almost complex maps $F:M \to S^6$ with image not contained in a totally geodesic $S^2$,
\emph{and}
\item strongly conformal harmonic maps $\varphi:M \to G_2/\SO{4}$ with
{\rm (a)} $\varphi^{\perp}$ strongly conformal, {\rm (b)} $\rank G'(\varphi) = 1$ and 
{\rm (c)} $G'(\varphi^{\perp}) \times G''(\varphi^{\perp})$ a non-constant subbundle,
\end{enumerate}
\emph{given by}
\be{f-phi}
F \mapsto \varphi = G''(f) \oplus f \oplus G'(f) 
\quad \text{with inverse} \quad
\varphi \mapsto F = \ii\, \ov{L} \times L \big/ \vert L \vert^2;
\ee
here $f = \spa\{F\}$ and $L$ is any nowhere zero (local) section of $G'(\varphi^{\perp})$.
\end{theorem}

\begin{proof} (i) Let $F$ be as in (1).  Since $f = \spa\{F\}$ is non-constant,
$G^{(-1)}(f)$, $f$ and $G^{(1)}(f)$ are non-zero and mutually orthogonal.
Since $f$ is almost complex, by Lemma \ref{le:assoc}
$\varphi = G^{(-1)}(f) \oplus f \oplus G^{(1)}(f)$ is a map into $G_2/\SO{4}$.
We have case (ii) of Proposition \ref{pr:tau}, and we use
the notation of that proposition writing $\psi _i = G^{(i)}(f)$ \ $(|i| \leq 2)$. 
We have $G'(\varphi) =  G^{(2)}(f) = \psi_2$ which is non-zero and
isotropic, so that $\varphi$ is strongly conformal
with $\rank G'(\varphi) = 1$.
In fact, combining subbundles in \eqref{tau}, we obtain the diagram

\be{tau-str-conf}
\xymatrix{ 
\psi_{-3} \oplus \psi_{-2}  \ar[r] & \psi_{-1} \oplus
\psi_0 \oplus \psi_1  \ar[r] & \psi_2 \oplus \psi_3  \ar@/_1.5pc/[ll]
}
\ee
which shows that $\varphi$ has the twistor lift $W:M \to T_1$
where $W = \psi_2 \oplus \psi_3$.  Further, $G'(\varphi^{\perp}) = \psi_{-1}$
is isotropic, so that $\varphi^{\perp}$ is strongly conformal; also 
$G'(\varphi^{\perp}) \times G''(\varphi^{\perp}) = \psi_{-1} \times \psi_1 = f$,
which is non-constant, so condition (2)(c) holds.

(ii) Conversely, let $\varphi$ be as in (2) and  set
$\alpha = G''(\varphi^{\perp})$. Since $\varphi^{\perp}$ is strongly
 conformal, $\alpha$ is isotropic; also, by Lemma \ref{le:2ffprops}(i),
 $\rank\,\alpha = \rank A''_{\varphi^{\perp}} = \rank A'_{\varphi} = 1$. 
For any nowhere zero (local) section $L$ of $\alpha$, set
$F = \ii\, \ov{L} \times L \big/ \vert L \vert^2$ --- note that this is
well defined under different choices of $L$ --- and set $f = \spa\{F\}$.
By condition (2)(c), $f$ is non-constant.   Further, we have 
 an orthogonal decomposition:
$\varphi = \ov{\alpha} \oplus f \oplus \alpha$.

Now $f$ is orthogonal to $\Ima A''_{\varphi^{\perp}}$, so is contained in
$\ker A'_{\varphi}$\,, hence, $G'(f) \subset \varphi$. But
$\ov{\alpha} = \Ima A'_{\varphi^{\perp}}$ is a holomorphic subbundle of $\varphi$,
so that $A''_{\ov{\alpha},f} = 0$; it follows
that $G'(f) \subset \alpha$.  
Since $f$ is non-constant, $G'(f) = \alpha$,  so we have the decomposition
$\varphi = G''(f) \oplus f \oplus G'(f)$\,.
Further, by Lemma \ref{le:assoc}, $F \times L = \ii L$ showing that $F$ is almost complex.

It is easily checked that the two transformations in \eqref{f-phi} are inverse.
\end{proof}

The above constructions give strongly conformal harmonic maps into $G_2/\SO{4}$ with Gauss bundles of rank $1$. To get ones with Gauss bundles of rank $2$, we adopt a different sort of construction, as follows.

\begin{proposition} \label{pr:constr} Let $F:M \to S^6$ be an almost complex map with image not contained in a totally geodesic $S^2$.
Set $\psi_3 = G^{(1)}(f) \!\times\! G^{(2)}(f)$ and $\psi_{-3} = \ov{\psi}_3$.
Then $\varphi = \psi_{-3} \oplus f \oplus \psi_3$  is a strongly conformal harmonic map $M \to G_2/\SO4$ with\/ $\rank\,G'(\varphi) = 2$.
\end{proposition}

\begin{proof}
Let $(\psi_i)$ be the $G_2$-flag \eqref{tau}
generated by $f$, i.e., $\psi_i = G^{(i)}(f)$ \ $(i=-2,\ldots,2)$,
$\psi_3 = \psi_1 \times \psi_2$, $\psi_{-3} = \ov{\psi}_3$.  Set
$\varphi = \psi_{-3} \oplus \psi_0 \oplus \psi_3$ and
$W = \psi_1 \oplus \psi_{-2}$.  By Lemma \ref{le:alg-S6}(i), $\varphi$ is associative, $W$ is complex-coassociative, and $\varphi$ has $J_2$-holomorphic lift
$W:M \to T_1$.  Further, $W = G'(\varphi) = G'(\psi_0) \oplus G'(\psi_{-3})$, which has rank $2$.
\end{proof}

\begin{example} \label{ex:torus}
Let $v_1 \in \l_{\alpha_1}$, $v_2 \in \l_{\alpha_1+\alpha_2}$ have norm $1/\sqrt{2}$ and set $v_3 = \ov{v_1 \times v_2}$\,.  Then $v_3$ also has norm
$1/\sqrt{2}$ and lies in $\l_{-2\alpha_1-\alpha_2}$, and we see that
$v_p \times v_q = \epsilon_{pqr}v_r$ and $v_p \times \ov v_q = (\ii/2)\delta_{pq}L_0$ where $L_0 \in \l_0$ has norm one.    As in
\cite[p.~420]{bolton-vrancken-woodward}, define a map $F:\cn \to S^6$ (a \emph{vacuum solution}, see \cite[\S 3]{burstall-pedit}) by
\be{torus}
F(z)
= \frac{1}{\sqrt{3}}\sum_{j=1}^3 \bigl(\eu^{\mu_j z - \ov{\mu}_j \zbar}v_j
+ \eu^{-\mu_j z + \ov{\mu}_j \zbar}\ov{v}_j \bigr)
\ee
where the $\mu_j$ are the three cube roots of unity.  Then $F$ is doubly periodic of periods $\pi$ and $\ii\pi/\sqrt{3}$, so factors to the torus
$T^2 = \cn/\langle \pi, \ii\pi/\sqrt{3} \rangle$. Further, it is 
 an almost complex map of type (III) into the $S^5$ orthogonal to
$L_0$, with harmonic sequence cyclic of order $6$.  By \cite[Corollary 6.4]{bolton-pedit-woodward}, this map is of
\emph{finite type}; so by \cite[Theorem 8]{pacheco-tori} it
is not of finite uniton number.
Note that the members of the harmonic squence $G^{(i)}(f)$ are spanned by the
$F_i$ defined in the proof of Proposition \ref{pr:tau}; for this particular example, the $F_i$ are equal to the successive derivatives
$\pa^i F/\pa z^i$ \ $(z \in \cn, \ i \in \nn)$.

Applying Proposition \ref{pr:alpha-constr} to $F$ gives nilconformal harmonic maps
$\varphi:T^2 \to G_2/\SO{4}$ by the formula \eqref{alpha-constr}.
Since $F$ is not of finite uniton number, by Remark
\ref{re:fi-uniton-no}, \emph{the maps $\varphi$ are not of finite uniton number
for any choice of\/ $\alpha$}.

The maps $\varphi$ may or may not be of \emph{finite type}.
For example, if $\alpha =  G'(f)$ where $f = \spa\{F\}$ so that
$\varphi$ is given by \eqref{f-phi}, then $\varphi$ \emph{is} of finite type by arguments similar to those in \cite[Theorem 9.2]{correia-pacheco-DPW} and
\cite[\S 4]{pacheco-tori}.
However, \emph{for most choices of $\alpha$, $\varphi$ is not of finite type}.
As a specific example, with $F$ given by \eqref{torus}, note that
$F_1 = \pa F/\pa z$ is a holomorphic section of $G'(f) \oplus G^{(2)}(f)$ and so of $F^{-1}T^{1,0}S^6$; let $Z$ be a meromorphic section of $G'(f) \oplus G^{(2)}(f)$ not lying in $G'(f)$.
Let $\alpha$ be the holomorphic line subbundle of $G'(f) \oplus G^{(2)}(f)$
spanned by $F_1 + b Z$ for some non-constant meromorphic function $b$
on $T^2$.  We may choose $b$ to have its zeros away from any poles of $Z$.
Then $\A(f) = \pi_{\alpha}^{\perp}G'(f)$; this is zero at points where $\alpha$ equals $G'(f)$, in particular, at the zeros of $b$.  It follows that the rank of $\A$ drops
at those points, so that by, for example, \cite[eqn.\ (9)]{pacheco-tori},
the map $\varphi$ cannot be of finite type.
The reader might like to compare this example with \cite[\S 6]{pacheco-tori} where a harmonic torus into quaternionic projective $2$-space $\HP{2}$ which is neither of finite type nor of finite uniton number is constructed.  
\end{example}

\subsection{Using superminimal harmonic maps into $S^6$} \label{subsec:superminimal}

Recall the diagram \eqref{diag:twistor-spaces}.  We define the projection
$ :Q^5 \to S^6$ by $\pi_6(\l) = \ii\, L \times \ov L \big/ |L|^2$ for any nowhere zero section $L$ of $\l$.
The following definition is given by R.~L.~Bryant \cite{bryant}.
\begin{definition} \label{def:superhor-Q5}
A full holomorphic map $h:M \to Q^5$ is called \emph{superhorizontal} if
it satisfies $h \times G'(h) = 0$, equivalently, any local holomorphic section
$H$ of $h$  satisfies $H \times H' = 0$.
\end{definition}  

Setting $\psi_i = G^{(3+i)}(h)$,  we have a map $\wh\psi:M \to F_2 = F_{2,3}^{\rn}$ defined by the formula 
$(\wh\psi_{-2},\wh\psi_{-1},\wh\psi_0,\wh\psi_1,\wh\psi_2)
= (\psi_{-3},\psi_{-2}\oplus\psi_{-1},\psi_0,\psi_1\oplus\psi_2,\psi_3)$.
Differentiating $h \times G'(h) = 0$ gives $h \times G^{(2)}(h) = 0$; it
follows that a full holomorphic map $h:M \to Q^5$ is superhorizontal
if and only if\/ $\wh\psi$ coincides with the unique $G_2$-flag
\eqref{F-filts}$(s$=2$)$ \emph{with $\l = \ov h$}.
Thus, under the identification of $Q^5$ with $\T_2$, 
\emph{a superhorizontal holomorphic map (Definition \ref{def:superhor-Q5})
$h:M \to Q^5$  corresponds to a superhorizontal holomorphic map
(Definition \ref{def:superhor}) $\wh\psi:M \to \T_2$}.
(Note that the harmonic sequence
$\psi =  (\psi_{-3},\psi_{-2},\psi_{-1},\psi_0,\psi_1,\psi_2,\psi_3)$
defines another superhorizontal map, this time into $\T_3$.) 

Let $F:M \to S^6$ be an almost complex map of type (I) (i.e., full and superminimal).   Then its
harmonic sequence \eqref{ha-seq} is a $G_2$-flag, in particular, $G^{(-3)}(f) \times G^{(-2)}(f) = 0$.  It follows that the first leg
$h = G^{(-3)}(f)$ of this flag is a full superhorizontal holomorphic map into
$Q^5$.  We are led to the following result of
Bryant \cite{bryant}, as described in \cite{bolton-woodward-campinas}.

\begin{theorem} \label{th:super-corr} There is a one-to-one correspondence between
\begin{enumerate}
\item full superhorizontal holomorphic maps $h:M \to Q^5$, and
\item almost complex maps $F:M \to S^6$ of type (I),
\end{enumerate}
given by $F = \ii\,\ov H \times H\big/|H|^2$, where $H$ is any (local) section of $h$, with inverse 
$h = G^{(-3)}(f)$ where $f = \spa\{F\}$.
\end{theorem}

\begin{remark}
(i) The harmonic sequence of $f$ is given by \eqref{ha-seq} with $G^{(i)}(f) = G^{(i+3)}(h)$.

(ii) Multiplication tables for the sections $F_i$ of this harmonic sequence defined by \eqref{Fi} are given in
\cite[Proposition 2]{martins} and \cite[Corollary 3.2]{fernandez}.

(iii)  For examples of full superhorizontal holomorphic maps
$h:S^2 \to Q^5$ given by simple polynomial formulae, see \cite{fernandez,martins}.
\end{remark}

 We can obtain harmonic maps into $G_2/\SO{4}$ from such $h$ as follows.

\begin{proposition} \label{prop:i=123}
Let $h:M \to Q^5$ be a full superhorizontal holomorphic map and let $F$ be the corresponding almost complex map,
as described in Theorem \ref{th:super-corr}.
Then $\varphi = G^{(-i)}(f) \oplus f \oplus G^{(i)}(f)$ is a harmonic map into 
$G_2/\SO{4}$ for $i = 1,2,3$.   Further, $\varphi$ is strongly conformal
$($i.e. $s(\varphi) = 1)$ for $i = 1,3$, but $s(\varphi) = 3$ for $i=2$.
\end{proposition}

\begin{proof}
For $i=1$, this is implied by Theorem \ref{th:one-one}.

For $i=2$, we have $\varphi = \pi_3 \circ \psi$ where $\psi = (\psi_{-3},\ldots,\psi_3):M \to \T_3$ is the harmonic sequence of $h$.  Since this is certainly
$J_2$-holomorphic, $\varphi$ is harmonic.  That $s(\varphi) = 3$ follows from
noting that $(\A)^4(\varphi^{\perp})$ contains $(\A)^4(\psi_{-3}) = \psi_1$, and so
is non-zero.	
 
For $i=3$, this is as a special case of Proposition \ref{pr:constr}. 
\end{proof}

\begin{remark} \label{re:two-lifts}  When $i=3$, $\varphi$ has the
$J_2$-holomorphic lift 
$(\psi_{-1} \oplus \psi_2, \varphi, \psi_1 \oplus \psi_{-2})$ into $\T_1$ as in Proposition \ref{pr:constr}; it also has the superhorizontal holomorphic lift
$(\psi_{-3}, \psi_{-2} \oplus \psi_{-1}, \psi_0, \psi_1 \oplus \psi_2, \psi_3)$
into $\T_2$, cf.\ Lemma \ref{le:uniqueness} and the remark following it.
\end{remark}

The above construction does not give harmonic maps with
$s(\varphi) = 2$.  To do this, we use Proposition \ref{pr:alpha-constr}
as follows.

\begin{example} \label{ex:s2}
Let $h:M \to Q^5$ be a full superhorizontal holomorphic map
 and let $F$ be the corresponding almost complex map.  Set $f = \spa\{F\}$, and let $\alpha$ be a holomorphic subbundle of $f^{\perp}$ which lies in
$T^{1,0}_f S^6 = h \oplus G^{(4)}(h) \oplus G^{(5)}(h) $ but is not contained in
either $h$ or $G^{(4)}(h) \oplus G^{(5)}(h)$.  Then the harmonic map defined by
\eqref{alpha-constr} has $s(\varphi) = 2$.

To see this, we may check that $(\A)^4(\varphi^{\perp}) = 0$.  Also
$(\A)^2(\varphi) \neq 0$, equivalently $G'(\varphi)$ is not isotropic:
indeed, it contains $A'_{\varphi}(\alpha)$ which has non-zero components in both
$G^{(-2)}(f)$ and $G^{(2)}(f)$ and so cannot be isotropic.

In fact, $\ov\alpha$ is a holomorphic subbundle of 
$\varphi$ lying in $\ker\A$ so that $\ov\alpha:M \to Q_5$ defines a
$J_2$-holomorphic lift of $\varphi$.   The corresponding flag
$\psi = (\psi_{-2}, \psi_{-1},\psi_0,\psi_1,\psi_2)$ has $\psi_2 = \ov\alpha$, $\psi_1 = G^{(2)}(f) \oplus \bigl\{\alpha^{\perp} \cap \bigl(G^{(-3)}(f) \oplus G^{(1)}(f)\bigr)\!\bigr\}$, $\psi_0 = f$.
It can be checked that $\psi$ is a $G_2$-flag with derivative $\psi_z$ in \eqref{hor-vert}, so defines a $J_2$-holomorphic lift into $\T_2$.  It is not a superhorizontal lift since, for example, $A'_{\psi_{-2},\psi_1}$ is non-zero.
See Example \ref{ex:s2-cont} for more on this example.
\end{example}

\section{Harmonic maps of finite uniton number} \label{sec:finite-uniton}
This section assumes some familiarity with the concepts of extended solution,
adding a uniton (also called `flag transform'),
finite uniton number, `the additional $S^1$-action' of C.-L. Terng as defined by Uhlenbeck \cite{uhlenbeck}, and Segal's Grassmannian model \cite{segal};
see \cite{J2} for a summary of what is needed here. In particular, a harmonic map
$\varphi:M \to G$ to a Lie group $G \subset \U{n}$
is said to be of \emph{finite uniton number}
if it has a polynomial extended solution $\Phi:M \to \Omega G$ into
the loop group of $G$ with $\Phi_{-1} = g\varphi$ for some $g \in G$.   Uhlenbeck \cite{uhlenbeck} shows that any polynomial extended solution
$\Phi$ has a factorization into unitons with values in $\Omega\U{n}$.
By replacing $\varphi$ by $g\varphi$, we shall assume that $g=e$, i.e., that $\Phi_{-1} = \varphi$.  

\subsection{Superhorizontal maps again} \label{subsec:superhor2} 
Let $\varphi:M \to G_2/\SO{4}$
be the projection of a superhorizontal holomorphic map
$\psi:M \to T_s$ \ $(s \in \{1,2,3\})$; as in \S \ref{subsec:superhor}, this is
given by the legs of \eqref{F-filts}, with $\ov\l$ and $\ov D$ holomorphic subbundles of $(\CC^7,\pa_{\zbar})$.  Then $\varphi$ is harmonic and an extended
 solution $\Phi$ of $\varphi$ can be read off from
\eqref{F-filts}, namely, in terms of the Grassmannian model $\W := \Phi\HH_+$\,,
\be{S1-limits}
\left. \begin{array}{rl}
(s = 1) \quad &
\W = \lambda^{-1} \ov W + W^{\perp} + \lambda \HH_+\,,
\\
(s = 2) \quad &
\W = \lambda^{-2} \ov{\l} + \lambda^{-1}\ov{\l}^a
	+ (\l^a)^{\perp} + \lambda \l^{\perp} + \lambda^2 \HH_+\,,
\\
(s = 3) \quad & \W = \lambda^{-3} \ov\l +
\lambda^{-2}\ov D 	+ \lambda^{-1}\ov{\l}^a + (\l^a)^{\perp} +
\lambda D^{\perp} 	+ \lambda^2 \l^{\perp} + \lambda^3 \HH_+\,.
\end{array} \right\}
\ee

\noindent 
According to Correia and Pacheco \cite{correia-pacheco-G2}, all $S^1$-invariant extended solutions
are given by this formula, so that \emph{a harmonic map $\varphi:M \to G_2/\SO{4}$ has a superhorizontal
holomorphic lift $\psi:M \to F_s$ if and only if it has an $S^1$-invariant extended solution}.
In fact, $\psi$ is the \emph{canonical lift} (see below)  of $\varphi$ defined by the extended solution.

\subsection{General harmonic maps of finite uniton number}
According to \cite{correia-pacheco-G2}, a harmonic map from a Riemann surface to
$G_2/\SO{4}$ of finite uniton number has an extended solution
$\Phi: M \to \Omega G_2$
with $\W = \Phi\HH_+$ satisfying (i) $\ov \W^{\perp} = \lambda \W$, (ii)
$\W$ is closed under the vector product $\times$, (iii) $\W$ satisfies
\be{W-limits}
\lambda^s \HH_+ \subset \W \subset \lambda^{-s} \HH_+
\ee
for some $s \in \{1,2,3\}$.
(In (ii), we have extended the vector product on $\cn^7$ to
$\HH = \sum_{i \in \zn} \lambda^i \cn^7$ by setting
$\lambda^i v \times \lambda^j w = \lambda^{i+j} v \times w$ \ $(v,w \in \cn^7, \ 
i,j \in \zn$).)

Condition (i) says that $\Phi$ takes values in the loop group $\Omega\SO{7}$, (ii) says that those values are actually in $\Omega G_2$, and (iii) follows from the assumption of finite uniton number.  We take the least $s$ for which
\eqref{W-limits} holds; then the \emph{$S^1$-invariant limit} \cite[\S 3.3]{unitons}, is given by \eqref{S1-limits} 
\cite{correia-pacheco-G2}.

Now let $\varphi:M \to \U{n}$ be a harmonic map of finite uniton number with an extended solution $\Phi$.  As above, set $\W = \Phi\HH_+$ and $s$ be the
least integer such that \eqref{W-limits} holds. Then (cf.\ \cite[\S 3]{unitons}) $\Phi$ gives an isomorphism
$\CC^7 \cong \HH_+/\lambda\HH_+ \to \W/\lambda\W$.
Noting that $\W \cap \lambda^{s+1}\HH_+ = \lambda^{s+1}\HH_+$ is
contained in $\lambda \W$, the filtration $\W = \W \cap \lambda^{-s}\HH_+ \subset \W \cap \lambda^{-s+1}\HH_+ \subset \cdots \subset \W \cap \lambda^s\HH_+  \subset \W \cap \lambda^{s+1}\HH_+$ descends to
a filtration of $\W/\lambda\W$ whose legs we denote by $A_i$, thus
$\W/\lambda\W = \sum_{i=-s}^s A_i$.
Then the \emph{canonical lift \cite[Theorem 4.8]{J2} of $\varphi$  defined by $\Phi$} is the map $\psi$ into a flag manifold with legs
$\psi_i = \Phi^{-1}(A_i)$. We can apply this to our situation as follows.

\begin{theorem} \label{th:fi-uniton}
Let $\varphi:M \to G_2/\SO{4}$ be a harmonic map of finite uniton number.
Let $\Phi: M \to G_2$ be an extended solution for $\varphi$ and let
$s \in \{1,2,3\}$ be the least integer such that $\W = \Phi\HH_+$
satisfies \eqref{W-limits}.
Then the canonical lift $\psi$ of\/ $\varphi$ defined by $\Phi$
gives a $J_2$-holomorphic lift of $\varphi$ into $\T_s$.
In particular, $\varphi$ is nilconformal with $s(\varphi) \leq s$.
\end{theorem}

\begin{proof}
  Since $\W$ is closed  under the vector product $\times$,
the latter descends to $\W/\lambda\W$.
 Then $\W/\lambda\W = \sum_{i=-s}^s A_i$ as above and, by
\cite[Proposition 6.6]{J2}, $\psi_{-i} = \ov \psi_i$, equivalently, $A_{-i} = \ov A_i \ 
\forall\, i$.  Note that the $A_i$ are non-zero and give the legs of the standard filtration \eqref{F-filts} in the $S^1$-invariant limit.  

In the general case, $\psi$ is a $G_2$-flag giving a lift into $\T_s$ if and only if
$A_i \times A_j \subset A_{i+j}$ (where we set $A_i = 0$ for $|i| > s$).
But $\W$ is closed under the vector product, so
$(\W \cap \lambda^i\HH_+) \times (\W \cap \lambda^j\HH_+) \subset 
\W \cap \lambda^{i+j}\HH_+$\,.  On applying the projection $\pi:\W \to \W/\lambda W$,
we deduce that $A_i \times A_j \subset \sum_{k \geq i+j}A_k$.
Replacing $i$, $j$ by $-i$, $-j$ and then taking the conjugate yields
$A_i \times A_j \subset \sum_{k \leq i+j}A_k$, hence $A_i \times A_j \subset A_{i+j}$ as required.
\end{proof}

\begin{remark} (i) It can be shown that $s(\varphi)$ \emph{is equal to} $s$ unless
$s=2$ when some special maps $\varphi$ have $s(\varphi) =1$,
see Remark \ref{re:two-lifts}.

(ii) If $\Phi$ is $S^1$-invariant, the canonical lift of $\varphi$ defined by $\Phi$ is given by the legs of \eqref{S1-limits} and is superhorizontal.
\end{remark}

\begin{example} \label{ex:s2-cont} 
Let $\varphi:M \to G_2/\SO{4}$ be the harmonic map with $s(\varphi) =2$
constructed in Example \ref{ex:s2}.
We show that $\varphi$ has uniton number $4$.  Indeed, starting with an extended solution for $f$ and multiplying by the two unitons $\alpha$ and $\ov\alpha$ shows that $\varphi$ has an extended solution $\Phi: M \to G_2$ given as a product of factors with values in $\Omega \SO{7}$ by
\be{ext-sol-s2}
\Phi = (\lambda^{-1}\pi_A + \pi_f + \lambda \pi_{\ov A})
			(\lambda^{-1}\pi_{\alpha} + \pi_{\alpha \oplus \ov\alpha}^{\perp} 	+ 		\lambda\pi_{\ov\alpha})
\ee
where $A = h_{(2)} = G^{(-3)}(f) \oplus G^{(-2)}(f) \oplus G^{(-1)}(f)$.

Note that $\lambda^2\Phi$ is a polynomial of degree exactly $4$ with
leading term $T_0 = \pi_A \circ \pi_{\alpha}$; this has image equal to $h$ which is full so that $\Phi$ is of \emph{type one} (see \cite[Remark 3.19]{unitons}).
This means that the \emph{(minimal) uniton number} of
$\varphi$ \cite[\S 13]{uhlenbeck} is
 \emph{exactly} $4$.  Further, since $\pi_{\alpha}$ and $\pi_A$ do not commute,
$\Phi$ is not $S^1$-invariant.

Let $H$ be a meromorphic section of $h$ and define $H_i$ iteratively by $H_0=H$,
$H_i = A'_{G^{(i-1)}(h)}(H_{i-1})$, so that $H_i$ is a holomorphic section
of $G^{(i)}(h)$. 
Let $\alpha $ be generated by $H_0 + t(a H_4 + b H_5)$ where $t$ is a non-zero
complex number (which we may take to be $1$), and 
$a$ and $b$ are meromorphic functions, not both zero.  In terms of
the Grassmannian model
$\W = \Phi\HH_+$, it can be checked that $\W/\lambda \W$ is generated by 
\be{basis-Ws2}
\lambda^{-2}H_0 +t(a H_4 +b H_5), \
\lambda^{-1}H_1 +t\lambda b H_6, \ \lambda^{-1}H_2 - \lambda t a H_6, \
H_3, \ \lambda H_4, \ \lambda H_5, \ \lambda^2 H_6\,.
\ee
Letting $t \to 0$, shows that the $S^1$-invariant limit is the type one
 extended solution of
$\varphi_0 = h \oplus G^{(3)}(h) \oplus G^{(6)}(h) = G^{(-3)}(f) \oplus f \oplus G^{(3)}(f)$;
this is the harmonic map of Proposition \ref{prop:i=123}
with $i = 3$; note that $\varphi_0$ has (minimal) uniton number $4$, which is equal to that of $\varphi$, however, \emph{$s(\varphi_0) = 1 < s(\varphi) =2$}.

Writing $(\lambda^2\Phi)^{-1} = S_0 + \lambda^{-1}S_1 + \cdots$,
we see that $S_0 = \pi_{\alpha} \circ \pi_A$. {}From \cite[Proposition 2.8(ii)]{unitons} the last `Uhlenbeck' uniton $\gamma_4$ in the alternating factorization \cite[Example 4.5 and \S 6.1]{unitons} is given by $\gamma_4 = \Ima S_0 = \alpha$,
so that \eqref{ext-sol-s2}
is the alternating factorization of $\Phi$. 
\end{example}

\begin{example} \label{ex:s=2} Again, let $h:M \to Q^5$ be a full superhorizontal 
holomorphic map, and let $b$ be a meromorphic function on $M$ which is not
 identically zero.  For any complex number $t$ and any local holomorphic section
$H$ of $h$ set $X = \lambda^{-3} \spa\{H + t \lambda^4 b H^{(5)}\}$
where $H^{(i)}$ denotes the $i$th derivative of $H$.
As in Guest \cite{guest-update}, set 
\be{W}
\W = X + \lambda X_{(1)} + \lambda^2 X_{(2)} + \cdots + \lambda^5 X_{(5)}
	+ \lambda^6 \HH_+
\ee
where $X_{(i)}$ denotes the subbundle of  spanned by the local holomorphic sections of $X$ and their derivatives  of order up to $i$.  

We show that $\W$ is globally well defined on $M$.  Indeed,
define the $H_i$ as in the last example.  Then it is not hard to see that
$\W/\lambda \W$ is spanned by
\be{basis-Ws3}
\lambda^{-3}(H + t \lambda^4 b H_5), \ \lambda^{-2}(H_1 + t \lambda^4 b H_6), \ \lambda^{i-3}H_i \ (i=2,\ldots,6).
\ee
Thus $\W$ is globally well defined, so there is an extended solution
$\Phi:M \to \Omega_{\alg}(\U{n})$ with $\W = \Phi\HH_+$\,. 

It can be checked that $\W$ satisfies the  condition
(i) $\ov{\W}^{\perp} = \lambda \W$.  Further, (ii) $\W$ is closed under the vector product; indeed, it can quickly be checked from the the multiplication table
for the $H_i$ in in \cite{martins} or \cite{fernandez}, that the multiplication table for the basis elements \eqref{basis-Ws3}
 in $\W/\lambda \W$ is isomorphic to that for
$\lambda^{i-3}H_i$ \ $(i = 0,\ldots,6)$.  (Similar remarks apply to the last example.)

It can be checked that the alternating factorization of the extended solution
$\Phi$ is
\[
\Phi \!=\! (\lambda^{-1} \pi_{g_{(2)}} + \pi_{\{g_{(2)} \oplus \ov{g_{(2)}}\}}^{\perp} + \lambda\pi_{\ov{g_{(2)}}})
	(\lambda^{-1} \pi_{g_{(1)}} + \pi_{\{g_{(1)} \oplus \ov{g_{(1)}}\}}^{\perp} + \lambda\pi_{\ov{g_{(1)}}}) 
	(\lambda^{-1}\pi_{\gamma_6} + \pi_{\gamma_6 \oplus \ov\gamma_6}^{\perp} +
\lambda \pi_{\ov\gamma_6})
\]
where $\gamma_6$, the last Uhlenbeck uniton, is spanned by
$H + tbH_5$.  
Again, we see this is the type one extended solution of $\varphi$, so that
the (minimal) uniton number of $\varphi$ is precisely $6$.
 
As $t \to 0$, $\Phi$ tends to the type one $S^1$-invariant extended solution
of the harmonic map $\varphi_0:G^{-2}(f) \oplus f \oplus G^{2}(f):M \to G_2/\SO{4}$
of Theorem 5.14($i=2$) where $f = G^{(3)}(h)$.  Note that $\varphi_0$ has
(minimal) uniton number $6$, which is equal to that of $\varphi$ and, in contrast to the last example, $s(\varphi_0) = s(\varphi) \ (=3)$.
\end{example}

\begin{remark}
We can describe the extended solutions $\Phi$ in the last two examples in terms of
the construction in \cite{burstall-guest} as follows. For Example
\ref{ex:s2-cont}, let $D$ be the discrete set where \eqref{basis-Ws2}
fails to be a basis.  Let $A:M \setminus D \to \Lambda^+_{\alg} \U{n}^{\cc}$ be given by the matrix with columns \eqref{basis-Ws2}.
Let $\{\xi_1,\xi_2\}$ be dual to the simple roots $\{\alpha_1,\alpha_2\}$ (see below).
Then $\Phi$ is given by multiplying $A$ by the
closed geodesic $\gamma_{\xi_1}: \lambda = \eu^{\ii t} \mapsto \exp(t\xi_1)$
and performing an Iwasawa deomposition.  A similar description holds for Example
\ref{ex:s=2}, this time setting $A$ equal to the matrix with
columns \eqref{basis-Ws3} and multiplying it
by the geodesic $\gamma_{\xi_1+\xi_2}$.
\end{remark}

\section{Lie theory proofs}\label{sec:Lie-theory}
\subsection{Canonical twistor fibrations} \label{subsec:can-twistor-fibr}
We summarize the theory of Burstall and Rawnsley \cite{burstall-rawnsley}, see also
\cite[\S 3]{burstall-guest}.
Let $G$ be a compact connected Lie group with Lie algebra $\g$; as in the above references we assume that $G$ has trivial centre, see \cite{correia-pacheco-can} for the case of non-trivial centre.  Write $\g^{\cc} = \g \otimes \cn$.\
Let $T$ be a maximal torus of $G$ and let $\t$ be its Lie algebra;
 then $\t^{\cc} = \t \otimes \cn \subset \g^{\cc}$
is the corresponding Cartan subalgebra.  Let $\{\alpha_1, \ldots, \alpha_{\l} \} \subset \ii\t$ be a choice of simple roots, and let
$\{\xi_1,\ldots,\xi_{\l}\} \subset \t$ be dual to the $\alpha_j$ in the sense that $\alpha_j(\xi_k) = \ii\delta_{jk}$.
A \emph{canonical element (for $G$)} is an element $\xi \in \t$ such that
$\alpha_j(\xi) = 0$ or $\ii$ for each simple root $\alpha_j$, equivalently, $\xi = \sum_{k \in I}\xi_k$
for some subset $I$ of $\{1,\ldots,\l\}$.  Note that $\alpha(\xi)\!/\ii$ is an integer for any root $\alpha$.

Let $\g_{\alpha} \subset \g^{\cc}$ denote the root space of $\alpha$.
Given a canonical element $\xi$, we obtain the following objects where the sums are over the roots $\alpha$
satisfying the given condition on the integer $\alpha(\xi)\!/\ii$:

(i) a \emph{parabolic subalgebra} $\q = \t^{\cc} + \sum_{\alpha(\xi)\!/\ii\,\geq 0} \g_{\alpha};$

(ii)  a \emph{(generalized) flag manifold} $F = G^{\cc}/Q$ where $Q$ is the parabolic subgroup with
 Lie algebra $\q$.  Set
$\h =\q \cap \g$ so that $\h^{\cc} = \t^{\cc} + \sum_{\alpha(\xi) = 0} \g_{\alpha}$; 
then $F = G/H$ where $H$ has Lie algebra $\h$ and is the centralizer of a torus (namely that with Lie algebra $\{X \in \t:\alpha_i(X) = 0 \ \forall\, i \notin I\}$).   The formulae in (i) and (ii) give all flag manifolds of $G;$

(iii) a \emph{complex structure $J_1$} on the flag manifold $F$ with $(1,0)$-space at the base point
(the identity coset $eH$) given by
$T^{1,0}_{J_1}F = \sum_{\alpha(\xi)\!/\ii\,> 0} \g_{\alpha};$

(iv) a \emph{symmetric decomposition} $\g = \k+\p$ where
$\k^{\cc} = \t^{\cc} + \sum_{\alpha(\xi)/\ii \text{ even}} \g_{\alpha},$ and
$\p^{\cc}= \sum_{\alpha(\xi)\!/\ii \text{ odd}} \g_{\alpha};$
choosing a subgroup $K$ of $G$ with Lie algebra $\k$ gives a \emph{symmetric space} $N = G/K$;

(v) a \emph{homogeneous fibration} $\pi: F = G/H \to N = G/K$ given by the inclusion of $H$ in $K$;
with $K$ connected, this is called the \emph{canonical fibration defined by $\xi$}.

(vi) an \emph{almost complex structure $J_2$} on $F$ obtained by reversing the orientation on the vertical space of this fibration; thus the horizontal and vertical $(1,0)$-spaces of $J_2$ at the base point are given by
\[
\Hh^{1,0}_{J_2}F = \Hh^{1,0}_{J_1}F = \sum_{\alpha(\xi)\!/\ii\,>0, \text{ odd}} \g_{\alpha}
\quad \text{and} \quad
\Vv^{1,0}_{J_2}F = \Vv^{0,1}_{J_1}F = \sum_{\alpha(\xi)\!/\ii\,<0, \text{ even}} \g_{\alpha}.
\]
Then Burstall and Rawnsley \cite[Corollary 5.10]{burstall-rawnsley} show that \emph{the homogeneous fibration $\pi:(F,J_2) \to N$ is a twistor fibration for harmonic maps}.

\subsection{Nilconformal maps} \label{subsec:nilconformal}
We adopt the following definition, which extends that given in Definition \ref{def:nilconformal}.
Let $\varphi:M \to G$ be a smooth map from a Riemann surface to a Lie group $G$.  As in \S \ref{subsec:Lie-groups}, on a local coordinate domain $(U,z)$, write
$A^{\varphi}_z = \frac{1}{2}\varphi^{-1}\varphi_z:U \to \g^{\cc}$.  Then we say that \emph{$\varphi$ is nilconformal if $A^{\varphi}_z$ is $\ad$-nilpotent},
in the sense that $\{\ad(A^{\varphi}_z)\}^k = 0$ for some $k \in \nn$.   Here we are thinking of $\ad(A^{\varphi}_z)$ as an endomorphism of the trivial vector bundle $U \times \g^{\cc}$;
the condition is clearly independent of choice of local coordinate and so makes sense for a map
$\varphi:M \to G$.
By a standard result on the Jordan--Chevalley decomposition, this is equivalent to
$\rho \circ A^{\varphi}_z$ being nilpotent for any finite-dimensional representation $\rho$ of $\g^{\cc}$.

We say that \emph{a map into a symmetric space $G/K$ is nilconformal if its composition
with the Cartan immersion of $G/K$ into $G$ is nilconformal}.
Then, with notation as before, we have
\begin{proposition} \label{pr:nilconformal}
Let $\pi:F=G/H \to N=G/K$ be a canonical fibration and $\psi:M \to F$ a $J_1$- or $J_2$-holomorphic map
from a Riemann surface. Then $\varphi = \pi \circ \psi$ is nilconformal.
\end{proposition}

\begin{proof}  As $\pi$ is a homogeneous fibration, it suffices to work at the base point.
Now, $\varphi_z$ is obtained by applying $\d\pi$ to the horizontal component of
$\psi_z$.  By part (vi) of \S \ref{subsec:can-twistor-fibr} this lies in
$\sum_{\alpha(\xi)\!/\ii\,>0} \g_{\alpha}$. This is nilpotent, so that $A^{\varphi}_z$ is $\ad$-nilpotent, as required.
\end{proof}

We remark that we only need to know that $\psi$ is `horizontally holomorphic', i.e., the horizontal part of $\psi_z$ intertwines the complex structure of $M$ with $J_1$ (equivalently $J_2$).

\subsection{Twistor fibrations for $G_2$} \label{subsec:twistor-fibr2}

We see how the theory of \S \ref{subsec:can-twistor-fibr} applies to $G_2$ and show Lemma
\ref{le:preserve-J2}, i.e. that the maps \eqref{inclusions} interwine both $J_1$ and $J_2$.
There are precisely three non-zero canonical elements $\xi$ for $G$; these define our three twistor spaces $T_s$ as follows.

(i) $\xi=\xi_2$.   Then the Lie subalgebra $\h$ defined in \S \ref{subsec:can-twistor-fibr}(ii) has complexification 
\[
\h^{\cc} = \t^{\cc}+\sum_{\alpha(\xi)=0}\g_{\alpha}=\t^{\cc}+\g_{\alpha_1}+\g_{-\alpha_1}\,.
\]
Thus $\h$ is a copy of $\u{2}$ which we denote by $\u{2}_+$; with $\U{2}_+$ the corresponding connected subgroup of $G_2$, the flag manifold that we get from the canonical element $\xi=\xi_2$ is exactly $T_1 = G_2/\U{2}_+$.
Now $\U{2}_+$ is the stabilizer in $G_2$ of the flag \eqref{std-flag1}($s$=1),
so that we have a $G_2$-equivariant isomorphism of fibre bundles  $T_1 \to \T_1$.
Regarding $T_1$ as the set of all rank $2$ complex-coassociative isotropic subspaces $W$ of $\cn^7$, this is given geometrically by \eqref{inclusions}($s$=1).

We now show that the map \eqref{inclusions}($s$=1) intertwines both $J_1$ and $J_2$.  By $G_2$-equivariance, it suffices to work at the base point $\psi$ of $\T_1$ given by the flag \eqref{std-flag1}($s$=1);
note that this corresponds to $W = \l_{\alpha_1+\alpha_2} \oplus \l_{2\alpha_1+\alpha_2}$\,.
{}From the theory in the last subsection, the $(1,0)$-tangent space of $T_1$ for the complex structure $J_1$
is given at that base point by 
\[
T^{1,0}_{J_1}T_1=\sum_{\alpha(\xi)\!/\ii\,>0}\g_{\alpha}=\g_{\alpha_2}+\g_{\alpha_1+\alpha_2}+\g_{2\alpha_1+\alpha_2}+\g_{3\alpha_1+\alpha_2}+\g_{3\alpha_1+2\alpha_2}.
\]
Now, for any geometric flag manifold $F = \U{n}\big/\U{d_0}\times\dots\times\U{d_t}$ we have $T^{1,0}_{J_1}F = \sum_{i < j}\Hom(\psi_i,\psi_j)$ (see \S\ref{subsec:twistor-Grass}).  Further, 
at the base point, the $\psi_i$ are the legs of the flag \eqref{std-flag1}($s=1$).
It is a standard fact from representation theory that the action of $\g$ on $\cn^7$ satisfies
$\g_{\alpha}(\l_{\beta}) \subset \l_{\alpha+\beta}$ for any roots $\alpha$ and $\beta$; it follows that a
positive root defines an element in $\sum_{i < j}\Hom(\psi_i, \psi_j)$.
This shows that the inclusion of $T_1$ in $F_1$ maps $T^{1,0}_{J_1}T_1$ into
 $T^{1,0}_{J_1}F_1$ and so intertwines the $J_1$'s, i.e., is $J_1$-holomorphic.

Now, the $(1,0)$-part of the horizontal and vertical spaces for $T_1$ with respect to $J_1$
are given at the base point by 
\[
\Hh^{1,0}_{J_1} T_1 = \sum_{\alpha(\xi)\!/\ii\,>0, \text{ odd}}\g_{\alpha} =\g_{\alpha_2} + \g_{\alpha_1+\alpha_2} + \g_{2\alpha_1+\alpha_2} + \g_{3\alpha_1+\alpha_2}
	\quad \text{and} \quad
\Vv^{1,0}_{J_1} T_1 = \g_{3\alpha_1+2\alpha_2}. 
\]
Comparing with \eqref{holo-tgt-space}, we see that these lie in $\Hh^{1,0}_{J_1} F_1$ and $\Vv^{1,0}_{J_1} F_1$,
respectively.  Hence the inclusion also intertwines $J_2$. 

(ii) $\xi=\xi_1$. Then the Lie subalgebra $\h$ defined in \S \ref{subsec:can-twistor-fibr}(ii) has complexification  
\be{h2}
\h^{\cc} = \t^{\cc}+\sum_{\alpha(\xi)=0}\g_{\alpha}=\t^{\cc}+\g_{\alpha_2}+\g_{-\alpha_2}\,.
\ee
Again, $\h$ is a copy of $\u{2}$ which we denote by $\u{2}_-$.
The corresponding connected subgroup $\U{2}_-$ is the stabilizer in $G_2$ of the flag
\eqref{std-flag2}($s$=2),
so that the flag manifold that we get from the canonical element $\xi=\xi_2$ is exactly
$T_2 = G_2/\U{2}_-$.
This flag does not project to the same element $\varphi \in G_2/\SO{4}$ as the flag \eqref{std-flag1}($s$=1).
To achieve that,
we must apply the element in the Weyl group given by reflection in the plane orthogonal to 
$3\alpha_1+\alpha_2$.  This replaces $\alpha_2$ by $3\alpha_1 + 2\alpha_2$ in \eqref{h2}, which gives a $\U{2}_-$ which is the stabilizer of the flag \eqref{std-flag1}($s$=2).
%(in contrast, note that there is no element of the Weyl group which transforms $\U{2}_-$ into the $\U{2}_+$ of case (i), as they are the stabilizers of inequivalent flags).

We do not apply any element of the Weyl group as is most convenient to use the $\u{2}_-$ given by \eqref{h2}.  Then we have a $G_2$-equivariant isomorphism of fibre bundles  $T_2 \to \T_2$.
Regarding $T_2$ as the set $Q^5$ of all rank $1$ isotropic subspaces $\l$ of $\cn^7$, this is given geometrically by \eqref{inclusions}($s$=2).
Note that $\psi \in \T_2$ also corresponds to $W = \psi_1$, a rank $2$ isotropic subspace of $\cn^7$
which is not complex-coassociative, justifying our alternative description of $T_2$ as the set of all such.

That the map \eqref{inclusions}($s$=2) intertwines both $J_1$ and $J_2$ is proved
as in case (i); this time, 
\[
\Hh^{1,0}_{J_1}T_2 =\sum_{\alpha(\xi)\!/\ii\,>0,\text{ odd}}\g_{\alpha} =\g_{\alpha_1}+\g_{\alpha_1+\alpha_2}+\g_{3\alpha_1+\alpha_2}+\g_{3\alpha_1+2\alpha_2}
\quad \text{and} \quad
\Vv^{1,0}_{J_1}T_2 = \g_{2\alpha_1+\alpha_2}.
\]

(iii)  $\xi = \xi_1+\xi_2$.  In this case, $\h = \t = \u{1} \oplus \u{1}$; hence, the flag manifold we obtain is 
$G_2/T = G_2/(\U{1} \times \U{1}) = T_3$.  Now $T$ is the stabilizer in $G_2$ of the flags
\eqref{std-flag1}($s=3$) and \eqref{std-flag2}($s=3$); as in case (ii), for convenience we work with the latter.
Then
\[
\Hh^{1,0}_{J_1}T_3 = \g_{\alpha_1}+\g_{2\alpha_1+\alpha_2}+\g_{3\alpha_1+2\alpha_2} \quad \text{and}
\quad \Vv^{1,0}_{J_1}T_3 = \g_{\alpha_1+\alpha_2} + \g_{3\alpha_1+\alpha_2}\,,
\]
and the proof that the map \eqref{inclusions}($s$=3) intertwines both $J_1$ and $J_2$
proceeds as before.

\subsection{Proof of Lemma \ref{le:J2-descr}}\label{subsec:J2-descr}

We need the following facts.
As in \S \ref{subsec:Grass}, we identify smooth maps $\varphi:M \to G_k(\cn^n)$ and
rank $k$ subbundles of $\CC^n$.  As described in that section, we give $\CC^n$ the connection
$D^{\varphi}$ and the Koszul--Malgrange structure with $\bar{\pa}$-operator given over each
coordinate domain $(U,z)$ by $D^{\varphi}_{\bar z}$. 
{}From \S \ref{subsec:Grass}, the connection $D^{\varphi}$ is the direct sum of the connections $\nabla_{\varphi}$ on $\varphi$
and $\nabla_{\varphi^{\perp}}$ on $\varphi^{\perp}$;
those connections are given by projection of the flat connection on $\CC^n$.

We now prove Lemma \ref{le:J2-descr} for $s=1$.  The cases $s=2$ and $s=3$ are similar.

Suppose that $\psi:M \to T_1$ is $J_2$-holomorphic. We need to show:
 \be{s=1}
W \text{ is a holomorphic subbundle of } (\CC^n, D^{\varphi}_{\zbar}) \text{ which lies in }
\ker A^{\varphi}_z.
\ee
By $J_2$-holomorphicity, the horizontal component of $\psi_z$ has values in\\ 
$\Hh^{1,0}_{J_2} T_1 = \Hh^{1,0}_{J_1} T_1 = \sum_{\alpha(\xi)\!/\ii\,>0, \text{ odd}}\g_{\alpha} =\g_{\alpha_2} + \g_{\alpha_1+\alpha_2} + \g_{2\alpha_1+\alpha_2} + \g_{3\alpha_1+\alpha_2}$, and the vertical 
component has values in $\Vv^{1,0}_{J_2} T_1 = \Vv^{0,1}_{J_1} T_1 = \g_{-3\alpha_1-2\alpha_2}$\,.

Hence the horizontal component of $\psi_z$ maps
$W = \psi_1 = \l_{\alpha_1+\alpha_2} \oplus \l_{2\alpha_1+\alpha_2}$ to $0$,
and the vertical component has zero component in $\End(\psi_i,\psi_j)$ for $j >i$\,
so that $A'_{\varphi^{\perp} \ominus W,W} = 0$.
As in \cite{burstall-wood}, this means $W$ is a holomomorphic subbundle of $(\CC^n, D^{\varphi}_{\zbar})$
and \eqref{s=1} follows.

Conversely, suppose that $\psi:M \to T_1$ is \emph{not} $J_2$-holomorphic.   Then
(i) \emph{either} the horizontal component of
$\psi_z$ has a non-zero component in one of the components of 
$\Hh^{0,1}_{J_2}T_1 = \g_{-\alpha_2} + \g_{-\alpha_1-\alpha_2} + \g_{-2\alpha_1-\alpha_2} + \g_{-3\alpha_1-\alpha_2}$
\emph{or} (ii) the vertical component of $\psi_z$ has a non-zero component in $\Vv^{0,1}_{J_2}T_1 = \g_{3\alpha_1+2\alpha_2}$.
In case (i), $A^{\varphi}_z(W)$ is non-zero: for example, if $\psi_z$ has
a non-zero component in $\g_{-\alpha_2}$, then
$A^{\varphi}_z(W) \supset A^{\varphi}_z(\l_{\alpha_1 + \alpha_2}) = \l_{\alpha_1} \neq 0$.
Similarly, in case (ii), $A'_{\ov{W}, W} = A'_{\psi_{-1},\psi_1}$ is non-zero, so that $W$ is not a holomorphic subbundle of
$(\CC^n, D^{\varphi}_{\zbar})$.  Hence one of the conditions  in \eqref{s=1} is violated.

\end{document}